\documentclass[a4paper,10pt]{article}   

\usepackage[utf8]{inputenc}         
\usepackage{graphicx}
\usepackage{subcaption} 
\usepackage{float}   
\usepackage{fancybox}		  
\usepackage{listings} 
\usepackage{amsfonts}
\usepackage{color}
\usepackage{amsmath}
\usepackage{balance}
\usepackage{amsmath}
\usepackage{amssymb}
\usepackage{amsbsy}
\usepackage{amsthm}
\usepackage{dsfont}
\usepackage{tikz}
\usepackage{appendix}

\usepackage{hyperref}  
\hypersetup{                    
    colorlinks=true,                
    breaklinks=true,                
    urlcolor= blue,                 
    linkcolor= red,                
    citecolor= blue                
    }

\usepackage{enumerate}
\usepackage{mathtools}
\usepackage[margin=1in]{geometry}
\usepackage{fancyhdr}
\usepackage{mathrsfs}

\definecolor{dgreen}{rgb}{0.0, 0.5, 0.0}
\definecolor{byzantium}{rgb}{0.44, 0.16, 0.39}

\newcommand{\R}{\mathbb{R}}
\newcommand{\N}{\mathbb{N}}
\newcommand{\elts}{\{1,\cdots,N\}}
\newcommand{\sumi}{\displaystyle\sum_{i=1}^N}
\newcommand{\sumj}{\displaystyle\sum_{j=1}^N}

\newcommand{\C}{\mathcal{C}}

\newcommand{\X}{\mathcal{X}}
\newcommand{\sign}{\text{sgn}}

\def\N {\mathbb{N}}

\def\R {\mathbb{R}}

\def\S {\mathbb{S}}

\newcommand{\Lip}{\mathrm{Lip}}

\renewcommand{\P}{\mathcal{P}}
\newcommand{\mt}{\mu_t}
\newcommand{\PR}{\P(\R^d)}
\newcommand{\MR}{\mathcal{M}(\R^d)}
\newcommand{\MsR}{\mathcal{M}^s(\R^d)}
\newcommand{\Wabp}{W^{a,b}_p}
\newcommand{\supp}{\mathrm{supp}}
\newcommand{\muo}{\mu_0}

\newcommand{\dt}{\Delta t}
\newcommand{\mutk}{\mu_t^k}
\newcommand{\tx}{\tilde{x}}
\newcommand{\tm}{\tilde{m}}

\newcommand{\schema}[1]{\b{\sc #1}}

\newcommand{\ba}{\phi}

\newcommand{\Kx}{K_{x_0}}
\newcommand{\tT}{\tilde{T}}

\newcommand{\bS}{\bar{S}}

\newtheorem{theorem}{Theorem}
\newtheorem{maintheorem}{Theorem}
\newtheorem{prop}{Proposition}

\newtheorem{corollary}{Corollary}
\newtheorem{definition}{Definition}
\newtheorem{lemma}{Lemma}

\newtheorem{hyp}{Hypothesis}
\newtheorem{rem}{Remark}

\newcommand{\Woo}{ W^{1,1}_1}
\newcommand{\WWoo}{\mathbb W^{1,1}_1}
\newcommand{\PcR}{\P_c(\R^d)}
\newcommand{\D}{\mathcal{D}}
\newcommand{\munk}{\mu_{n\dt}^k}
\newcommand{\munnk}{\mu_{(n+1)\dt}^k}
\newcommand{\munkk}{\mu_{n\dt}^{k+1}}
\newcommand{\munnkk}{\mu_{(n+1)\dt}^{k+1}}
\newcommand{\mundk}{\mu_{(n+\frac12)\dt}^k}
\newcommand{\mundkk}{\mu_{(n+\frac12)\dt}^{k+1}}
\newcommand{\nunk}{\nu_{n}^k}

\newcommand{\nunkk}{\nu_{n}^{k+1}}

\newcommand{\nundk}{\nu_{n+\frac12}^k}
\newcommand{\nundkk}{\nu_{n+\frac12}^{k+1}}
\newcommand{\Hnk}{H_{n}^k}

\newcommand{\Hnkk}{H_{n}^{k+1}}

\newcommand{\Hndkk}{H_{n+\frac12}^{k+1}}
\newcommand{\dtt}{\frac{\dt}{2}}
\newcommand{\bmu}{\bar \mu}

\newcommand{\Rnk}{R_{n,k}}
\newcommand{\Cc}{C_{\mathrm{c}}}
\newcommand{\B}{\mathcal{B}}
\newcommand{\phiR}{\phi_R}
\newcommand{\La}{L_\phi}
\newcommand{\mutv}{|\mu|}
\DeclarePairedDelimiter\TV{|}{|}
\newcommand{\mus}{\mu^*}
\newcommand{\mun}{\mu_n}
\newcommand{\muunk}{\mu_{n_k}}

\newcommand{\M}{M_{\minit}}
\newcommand{\Km}{K_{\minit}}
\newcommand{\xinit}{x^\mathrm{in}}
\newcommand{\minit}{m^\mathrm{in}}

\newif\ifNikitaiscool
\Nikitaiscooltrue
\newif\ifAppendix
\Appendixtrue

\title{\LARGE \bf
Mean-field limit of collective dynamics with time-varying weights
}

\author{
Nastassia Pouradier Duteil\thanks{Sorbonne Universit\'e, Inria, Universit\'e Paris-Diderot SPC, CNRS, Laboratoire Jacques-Louis Lions, Paris, France} }

\begin{document}

\date{}
\maketitle

\thispagestyle{plain}
\pagestyle{plain}
\bibliographystyle{abbrv}

\abstract{In this paper, we derive the mean-field limit of a collective dynamics model with time-varying weights, for weight dynamics that preserve the total mass of the system as well as indistinguishability of the agents. 
The limit equation is a transport equation with source, where the (non-local) transport term corresponds to the position dynamics, and the (non-local) source term comes from the weight redistribution among the agents.
We show existence and uniqueness of the solution for both microscopic and macroscopic models and introduce a new empirical measure taking into account the weights. We obtain the convergence of the microscopic model to the macroscopic one by showing continuity of the macroscopic solution with respect to the initial data, in the Wasserstein and Bounded Lipschitz topologies.
}


\section*{Introduction}

A wide range of mathematical models fall into the category of \textit{interacting particle systems}.
Whether they describe the trajectories of colliding particles \cite{Degond04}, the behavior of animal groups \cite{Aoki82,CS07,G08,VCABC95}, the cooperation of robots \cite{B09} or the evolution of opinions \cite{DG74,F56,HK02},
their common objective is to model the dynamics of a group of particles in interaction.
Some of the most widely used models include the Hegselmann-Krause model for opinion dynamics \cite{HK02}, the Vicsek model for fish behavior \cite{VCABC95} and the Cucker-Smale model for bird flocks \cite{CS07}.
Two main points of view can be adopted in the modeling process. The Lagrangian (or microscopic) approach deals with individual particles and models the trajectory of each one separately, via a system of coupled Ordinary Differential Equations (ODE).
This approach's major limitation is that the dimension of the resulting system is proportional to the number of particles, which can quickly become unmanageable. 
To combat this effect,
one can instead adopt the Eulerian (or macroscopic) approach, and track the concentration of particles at each point of the state space. The resulting equation is a Partial Differential Equation (PDE) giving the evolution of the density of particles over the state space, and whose dimension is independent of the number of particles.

The question of how microscopic properties of particles give rise to macroscopic properties of the system is fundamental in physics.
A way to connect the microscopic and the macroscopic points of view is through the \textit{mean-field limit}.
First introduced in the context of gas dynamics, the mean-field limit, applied to systems of weakly interacting particles with a large radius of interaction, derives the macroscopic equation as the limit of the microscopic one when the number of particles tends to infinity \cite{braun1977, Dobrushin79}. 
The term \textit{mean-field} refers to the fact that the effects of all particles located at the same position are averaged, instead of considering the individual force exerted by each particle.
The mean-field limits of the Hegselmann-Krause, Vicsek and Cucker-Smale models were derived in \cite{Dobrushin79, CFT08, DM08, HT08}.
More specifically, the mean-field limit of a general system of interacting particles described by 
\begin{equation}\label{eq:syst-micro-gen}
\displaystyle \dot{x}_i(t) = \frac{1}{N}\sum\limits_{j=1}^N  \ba(x_j(t)-x_i(t)) \\
\end{equation}
is given by the non-local transport equation in the space of probability measures
\begin{equation}\label{eq:transport-gen}
\partial_t \mt(x) + \nabla\cdot \left(V[\mu_t](x) \mt(x)\right) = 0, \qquad V[\mu_t](x) = \int_{\R^d} \ba(y-x) d\mu_t(y),
\end{equation}
where $\mu_t(x)$ represents the density of particles at position $x$ and time $t$, and where the velocity $ V[\mu_t]$ is given by convolution with the density of particles.
The proof of the mean-field limit lies on the key observation that the \textit{empirical measure}
$
\mu_t^N = \frac{1}{N}\sum_{i=1}^N \delta_{x_i(t)},
$
defined from the positions of the $N$ particles satisfying the microscopic system \eqref{eq:syst-micro-gen},
is actually a solution to the macroscopic equation \eqref{eq:transport-gen}.
Notice that the passage from the microscopic system to its macroscopic formulation via the empirical measure entails an irreversible information loss. Indeed, the empirical measure keeps track only of the number (or proportion) of particles at each point of space, and loses the information of the indices, that is the ``identity'' of the particles.
This observation illustrates a necessary condition for the mean-field limit to hold: the \textit{indistinguishability} of particles. Informally, two particles $x_i$, $x_j$ are said to be indistinguishable if they can be exchanged without modifying the dynamics of the other particles. System \eqref{eq:syst-micro-gen} satisfies trivially this condition, since the interaction function $\phi$ depends only on the positions of the particles and not on their indices.

In \cite{McQuadePiccoliPouradierDuteil19, PiccoliPouradierDuteil20}, we introduced an augmented model for opinion dynamics with time-varying influence. 
In this model, each particle, or agent, is represented both by its opinion $x_i$ and its weight of influence $m_i$.
The weights are assumed to evolve in time via their own dynamics, and model a modulating social hierarchy within the group, where the most influential agents (the ones with the largest weights) have a stronger impact on the dynamics of the group.
The microscopic system is written as follows:
\begin{equation}\label{eq:syst-micro-intro}
\displaystyle \dot{x}_i(t) = \frac{1}{M} \sum\limits_{j=1}^N m_j(t) \ba(x_j(t)-x_i(t)), \qquad
\displaystyle \dot{m}_i(t) = \psi_i\Big((x_j(t))_{j\in\elts},(m_j(t))_{j\in\elts}\Big),
\end{equation}
where the functions $\psi_i$ give the weights' dynamics.

As for the classical dynamics \eqref{eq:syst-micro-gen}, we aim to address the natural question of the large population limit.
To take into account the weights of the particles, we can define a modified empirical measure by
$
\mu_t^N = \frac{1}{M}\sum_{i=1}^N m_i(t) \delta_{x_i(t)},
$
so that $\mu_t^N(x)$ represents the \textit{weighted proportion} of the population with opinion $x$ at time $t$. 
In this new context, informally, indistinguishability is satisfied if agents $(x_i,m_i)$ and $(x_j,m_j)$ can be exchanged or grouped without modifying the overall dynamics. 
However, this condition may or may not be satisfied, depending on the weight dynamics $\psi_i$ in the general system \eqref{eq:syst-micro-intro}.
In \cite{AyiPouradierDuteil20}, we derived the graph limit of system \eqref{eq:syst-micro-intro} for a general class of models in which indistinguishability is not necessarily satisfied.
Here, on the other hand, in order to derive the mean-field limit of system \eqref{eq:syst-micro-intro}, we will focus on a subclass of mass dynamics that does preserve indistinguishability of the particles, given by:
\begin{equation}\label{eq:massdyn-intro}
\psi_i(x,m) := m_i \frac{1}{M^q} \sum_{j_1=1}^N\cdots\sum_{j_q=1}^N m_{j_1} \cdots m_{j_q} S(x_i, x_{j_1},\cdots x_{j_q}).
\end{equation}
Given symmetry assumptions on $S$, this specific choice of weight dynamics ensures that the weights remain positive, and also preserves the total weight of the system. From a modeling point of view, since the weights represent the agents' influence on the group, it is natural to restrict them to positive values. The total weight conservation implies that no weight is created within the system, and that the only weight variations are due to redistribution. 
One can easily prove that if $(x_i,m_i)_{i\in\elts}$ satisfy the microscopic system \eqref{eq:syst-micro-intro}-\eqref{eq:massdyn-intro}, the modified empirical measure $\mu_t^N$ satisfies the following transport equation with source 
\begin{equation}\label{eq:transportsource-intro}
\partial_t \mt(x) + \nabla\cdot \left(V[\mu_t](x) \mt(x)\right) = h[\mt](x),
\end{equation}
in which the left-hand part of the equation, representing non-local transport, is identical to the limit PDE \eqref{eq:transport-gen} for the system without time-varying weights. The non-local source term of the right-hand side corresponds to the weight dynamics and is given by convolution with $\mu_t$: 
\[
 h[\mu_t](x) = \left(\int_{(\R^d)^q} S(x,y_1,\cdots,y_q) d\mu_t(y_1)\cdots d\mu_t(y_q)\right) \mu_t(x).
\]
Since we impose no restriction on the sign of $S$, this source term $h[\mt]$ belongs to the set of signed Radon measures, even if (as we will show), $\mu_t$ remains a probability measure at all time.

In \cite{PiccoliRossiTournus20}, well-posedness of \eqref{eq:transportsource-intro} was proven for a globally bounded source term satisfying a global Lipschitz condition with respect to the density $\mu_t$. However, the possibly high-order non-linearity of our source term $h[\mt]$ prevents us from applying these results in our setting.

Thus, the aim of this paper is to give a meaning to the transport equation with source \eqref{eq:transportsource-intro}, to prove existence and uniqueness of its solution, and to show that it is the mean-field limit of the microscopic system \eqref{eq:syst-micro-intro}-\eqref{eq:massdyn-intro}.
Our central result can be stated as follows: 
\begin{theorem}\label{th:convergence}
For each $N\in\N$, 
let $(x^N_i, m^N_i)_{i\in\elts}$ be the solutions to \eqref{eq:syst-micro-intro}-\eqref{eq:massdyn-intro} on $[0,T]$,
and let $\mu^N_t:= \frac{1}{M}\sum_{i=1}^N m_i^N(t) \delta_{x_i^N(t)}$ be the corresponding empirical measures.
If there exists $\mu_0 \in \PcR$ such that 
$
\lim_{N\rightarrow\infty} \mathscr{D}(\mu^N_0,\mu_0) = 0, 
$
then for all $t\in [0,T]$, 
$$
\lim_{N\rightarrow\infty} \mathscr{D}(\mu^N_t,\mu_t) = 0, 
$$
where $\mu_t\in\PcR$ is the solution to the transport equation with source \eqref{eq:transportsource-intro}.
\end{theorem}
The convergence holds in the Bounded Lipschitz and in the Wasserstein topologies, where $\mathscr{D}$ represents either the Bounded Lipschitz distance, or any of the $p$-Wasserstein distances ($p\in\N^*$).
In particular, we show that the solution stays a probability measure at all time, a consequence of the total mass conservation at the microscopic level.

We begin by presenting the microscopic model, and by showing that under key assumptions on the mass dynamics, it preserves not only indistinguishability of the agents, but also positivity of the weights as well as the total weight of the system.
We then recall the definition and relationship between the Wasserstein, Generalized Wasserstein and Bounded Lipschitz distances. 
The third section is dedicated to the proof of existence and uniqueness of the solution to the macroscopic equation, by means of an operator-splitting numerical scheme. We show continuity with respect to the initial data in the Bounded Lipschitz and Wasserstein topologies. This allows us to conclude with the key convergence result, in Section 4. 
Lastly, we illustrate our results with numerical simulations comparing the solutions to the microscopic and the macroscopic models, for a specific choice of weight dynamics.

\section{Microscopic Model}

In \cite{McQuadePiccoliPouradierDuteil19}, a general model was introduced for opinion dynamics with time-varying influence. Given a set of $N$ agents with positions $(x_i)_{i\in\elts}$ and weights $(m_i)_{i\in\elts}$, an agent $j$ influences another agent $i$'s position (or opinion) depending on the distance separating $i$ and $j$, as well as on the weight (or ``influence'') of $j$. In parallel, the evolution of each agent's weight $m_j$ depends on all the agents' positions and weights. In this general setting, the system can be written as: 

\begin{equation}\label{eq:syst-microgen}
\begin{cases}
\displaystyle \dot{x}_i(t) = \frac{1}{M} \sum\limits_{j=1}^N m_j(t) \ba(x_j(t)-x_i(t)), \\
\displaystyle \dot{m}_i(t) = \psi_i\Big((x_j(t))_{j\in\elts},(m_j(t))_{j\in\elts}\Big),
\end{cases}
\quad i\in\elts,
\end{equation}
where $M = \sum_{i=1}^N m_i^0$ represents the initial total mass of the system, $\phi\in C(\R^{dN}; \R^{dN})$ denotes the interaction function and $\psi_i\in C(\R^{dN}\times\R^{N}; \R)$ dictates the weights' evolution.
Well-posedness of this general system was proven in \cite{AyiPouradierDuteil20}, for suitable weight dynamics $\psi_i$.

In this paper, we aim to study the mean-field limit of system \eqref{eq:syst-microgen} for a more specific choice of weight dynamics that will ensure the following properties: 
\begin{itemize}
\item positivity of the weights: $m_i\geq 0$ for all $i\in\elts$;
\item conservation of the total mass: $\sum_{i=1}^N m_i \equiv M$;
\item indistinguishability of the agents.
\end{itemize}
These key properties will be used extensively to prove well-posedness of the system and convergence to the mean-field limit. We now introduce the model that will be our focus for the rest of the paper.
Let $(x_i^0)_{i\in\elts}\in\R^{dN}$ and $(m_i^0)_{i\in\elts}\in (\R^+)^N$. We study the evolution of the $N$ positions and weights according to the following dynamics:
\begin{equation}\label{eq:syst-micro}
\begin{cases}
\displaystyle \dot{x}_i(t) = \frac{1}{M} \sum\limits_{j=1}^N m_j(t) \phi \left(x_j(t)-x_i(t)\right), \quad x_i(0)=x_i^0 \\
\displaystyle \dot{m}_i(t) = m_i(t)\frac{1}{M^q} \sum_{j_1=1}^N \cdots \sum_{j_q=1}^N m_{j_1}(t)\cdots m_{j_q}(t) S(x_i(t), x_{j_1}(t), \cdots x_{j_q}(t)), \quad m_i(0)=m_i^0
\end{cases}
\end{equation}
where $\ba $
and $S$ satisfy the following hypotheses:
\begin{hyp}\label{hyp:abar}
 $\ba\in \Lip(\R^d;\R^d)$ with $\|\ba\|_\Lip := L_\phi$. 
\end{hyp}
\begin{rem}
The most common models encountered in the literature use an interaction function $\ba$ of one of the following forms: 
\begin{itemize}
\item $\ba(x):= a(|x|)x$ for some $a:\R^+\rightarrow\R$ 
\item $\ba(x):=\nabla W(x)$ is the gradient of some interaction potential $W:\R^d\rightarrow\R$.
\end{itemize}
\end{rem}
\begin{hyp}\label{hyp:S}
 $S\in C((\R^d)^{q+1}; \R)$ is 
 globally bounded and 
 Lipschitz. More specifically, there
 exist $\bS$, 
 $L_S>0$ such that 
\begin{equation}\label{eq:psiSkbound}
 \forall y\in(\R^d)^{q+1}, \quad |S(y)|\leq \bS.
\end{equation}
and  
\begin{equation}\label{eq:psiSklip}
 \forall y\in (\R^d)^{q+1}, \forall z\in (\R^d)^{q+1}, | S(y_0,\cdots,y_q) - S(z_0,\cdots,z_q) | \leq L_S \sum_{i=0}^q |y_i-z_i|.
\end{equation}
Furthermore, we require that $S$ satisfy the following skew-symmetry property: there exists $(i,j)\in \{0,\cdots,q\}^2$ such that for all $y\in(\R^d)^{q+1}$, \;
\begin{equation}\label{eq:condS}
  S(y_0,\cdots, y_i,\cdots,y_j,\cdots,y_q) =-S(y_0,\cdots, y_j,\cdots,y_i,\cdots,y_q) .
\end{equation}
\end{hyp}
\begin{rem}
The global boundedness of $S$ \eqref{eq:psiSkbound} is assumed to simplify the presentation, but all our results also hold without this assumption. Indeed, the continuity of $S$ is enough to infer the existence of a global bound $S_R$ as long as all $x_i$ are contained in the ball $S(0,R)$, or, in the macroscopic setting, as long as $\supp(\mu)\subset B(0,R)$.
\end{rem}
The skew-symmetric property of $S$ is essential in order to prevent blow-up of the individual weights. Indeed, as we show in the following proposition, it allows us to prove that the total mass is conserved and that each of the weights stays positive. Thus, despite the non-linearity of the weight dynamics, the weights remain bounded at all time, and in particular there can be no finite-time blow-up, which will ensure the existence of the solution. 
\begin{prop}\label{Prop:m}
Let $(x,m)\in C([0,T];(\R^d)^N\times\R^{N})$ be a solution to \eqref{eq:syst-micro}. Then it holds:
\begin{enumerate}[(i)]
\item For all $t\in [0,T]$, $ \sum_{i=1}^N m_i(t) = M.$
\item If for all $i\in\elts$, $m^0_i>0$, then for all $t\in [0,T]$,
for all $ i\in \elts$, $m_i(t)>0$.
\item If for all $i\in\elts$, $m^0_i>0$, then for all $t\in [0,T]$,
for all $ i\in \elts$, $m_i(t)\leq m^0_i e^{\bS t}$.
\end{enumerate}
\end{prop}
\begin{proof}
We begin by proving that the total mass $\sum_i m_i$ is invariant in time.
Without loss of generality, let us suppose that for all $y\in(\R^d)^{q+1}$,
\[
S(y_0, y_1,\cdots,y_q) =-S(y_1,y_0,\cdots,y_q).
\]
It holds
\[
\begin{split}
\frac{d}{dt}\sum_{i=1}^N m_i = & \sum_{j_0=1}^N m_{j_0}\frac{1}{M^q} \sum_{j_1=1}^N \cdots \sum_{j_q=1}^N m_{j_1}\cdots m_{j_q} S(x_{j_0}, x_{j_1}, \cdots x_{j_q}) \\
 = & \frac{1}{M^q} \sum_{j_0<j_1} \sum_{j_2=1}^N \cdots \sum_{j_q=1}^N m_{j_0} m_{j_1}\cdots m_{j_q} S(x_{j_0}, x_{j_1}, \cdots x_{j_q}) \\
& + \frac{1}{M^q} \sum_{j_0>j_1} \sum_{j_2=1}^N \cdots \sum_{j_q=1}^N m_{j_0} m_{j_1}\cdots m_{j_q} S(x_{j_0}, x_{j_1}, \cdots x_{j_q})\\
= & \frac{1}{M^q} \sum_{j_0<j_1} \sum_{j_2=1}^N \cdots \sum_{j_q=1}^N m_{j_0} m_{j_1}\cdots m_{j_q} S(x_{j_0}, x_{j_1}, \cdots x_{j_q}) \\
& + \frac{1}{M^q} \sum_{j_1>j_0} \sum_{j_2=1}^N \cdots \sum_{j_q=1}^N m_{j_1} m_{j_0}\cdots m_{j_q} S(x_{j_1}, x_{j_0}, \cdots x_{j_q})
=0
\end{split}
\]
by the antisymmetry property of $S$.

Let us now suppose that $m_i^0>0$ for all $i\in\elts$.
Let $t^*:=\inf\{t\geq 0\, | \, \exists i\in \elts, \; m_i(t)=0\}$.
Assume that $t^*<\infty$. Then for all $i\in\elts$, for all $t<t^*$, 
\[
\begin{split}
\dot m_i =  m_i \frac{1}{M^q} \sum_{j_1=1}^N \cdots \sum_{j_q=1}^N m_{j_1}\cdots m_{j_q} S(x_{i}, x_{j_1}, \cdots x_{j_q})
\geq  - m_i \frac{1}{M^q} \sum_{j_1=1}^N \cdots \sum_{j_q=1}^N m_{j_1}\cdots m_{j_q} \bS = - \bS m_i,
\end{split}
\]
where the last equality comes from the first part of the proposition.
From Gronwall's Lemma, for all $t<t^*$, it holds
\[ 
m_i(t) \geq m_i^0 e^{-\bS t} \geq m_i^0 e^{-\bS t^*}>0. 
\]
Since $m_i$ is continuous, this contradicts the fact that there exists $i\in\elts$ such that $m_i(t^*)=0$. Hence for all $t\geq 0$, $m_i(t)>0$.

Lastly, the third point is a consequence of Gronwall's Lemma.
\end{proof}
Well-posedness of the system \eqref{eq:syst-micro} is a consequence of the boundedness of the total mass. We have the following result. 
\begin{prop}\label{Prop:WellPosmicro}
For all $T>0$, there exists a unique solution to \eqref{eq:syst-micro} defined on the interval $[0,T]$.
\end{prop}
\begin{proof}
The proof, modeled after the proofs for the well-posedness of the Graph Limit model in \cite{AyiPouradierDuteil20}, is provided in Appendix \ref{Sec:ExUniqMicro}.
\end{proof}

We draw attention to the fact that System \eqref{eq:syst-micro} also preserves indistinguishability of the agents. This property, introduced in \cite{PiccoliPouradierDuteil20} and \cite{AyiPouradierDuteil20}, is necessary for the definition of empirical measure to make sense in this new setting.

Indeed, the empirical measure, defined by $\mu^N_t = \frac{1}{M}\sum_{i=1}^N m_i^N(t) \delta_{x_i^N(t)}$ is invariant by relabeling of the indices or by grouping of the agents. Hence for the macroscopic model to reflect the dynamics of the microscopic one, the microscopic dynamics must be the same for relabeled or grouped initial data. This leads us to the following \textit{indistinguishability} condition:

\begin{definition}\label{Def:indist}
We say that system \eqref{eq:syst-microgen} satisfies \emph{indistinguishability} if
for all $J \subseteq \elts$, 
for all $(x^0, m^0)\in \R^{dN}\times\R^N$ and $(y^0, p^0)\in \R^{dN}\times\R^N$ satisfying 
\begin{equation*}
\begin{cases}
x_i^0 = y_i^0 = x_j^0 = y_j^0 \qquad \text{ for all } (i,j)\in J^2 \\
x_i^0 = y_i^0 \qquad \text{ for all } i\in \elts \\
m_i^0 = p_i^0 \qquad \text{ for all } i\in J^c \\
\sum_{i\in J} m_i^0 = \sum_{i\in J} p_i^0,
\end{cases}
\end{equation*}
the solutions $t\mapsto (x(t),m(t))$ and $t\mapsto (y(t),p(t))$ to system \eqref{eq:syst-microgen} with respective initial conditions
$(x^0, m^0)$ and $(y^0, p^0)$ satisfy for all $t\geq 0$,
\begin{equation*}
\begin{cases}
x_i(t) = y_i(t) = x_j(t) = y_j(t) \qquad \text{ for all } (i,j)\in J^2 \\
x_i(t) = y_i(t) \qquad \text{ for all } i\in \elts \\
m_i(t) = p_i(t) \qquad \text{ for all } i\in  J^c \\
\sum_{i\in J} m_i(t) = \sum_{i\in J} p_i(t).
\end{cases}
\end{equation*}
\end{definition}
Whereas the general system \eqref{eq:syst-microgen} does not necessarily satisfy this property, one easily proves that system \eqref{eq:syst-micro} does satisfy indistinguishability (see \cite{AyiPouradierDuteil20} for the detailed proof).

\section{Notations and distances}

Let $\PR$ denote the set of probability measures of $\R^d$, $\PcR$ the set of probability measures with compact support, $\MR$ the set of (positive) Borel measures with finite mass, and $\MsR$ the set of signed Radon measures. Let $\B(\R^d)$ denote the family of Borel subsets of $\R^d$.

From here onward, $C(E)$ (respectively $C(E;F)$) will denote the set of continuous functions of $E$ (resp. from $E$ to $F$), $C^\Lip(E)$ (respectively $C^\Lip(E;F)$) the set of Lipschitz functions, and $\Cc$ (respectively $\Cc(E;F)$) the set of  functions with compact support. The Lipschitz norm of a function $f\in C^\Lip(E;F)$ is defined by
\[
\|f\|_\Lip := \sup_{x,y\in E,x\neq y}\frac{d_F(f(x)-f(y))}{d_E(x-y)}.
\]

For all $\mu\in \MR$, we will denote by $|\mu|:=\mu(\R^d)$ the total mass of $\mu$. 

For all $\mu\in\MsR$, let $\mu_+$ and $\mu_-$ respectively denote the upper and lower variations of $\mu$, defined by 
$\mu_+(E):=\sup_{A\subset E} \mu(A)$ and $\mu_-(E):=- \inf_{A\subset E} \mu(A)$ for all $E\in\B(\R^d)$, so that $\mu = \mu_+ -\mu_-$. 
We will denote by $\mutv$ the total variation of $\mu$ defined by $\mutv:= \mu_+(\R^d)+\mu_-(\R^d)$. 

\subsection{Generalized Wasserstein and Bounded Lipschitz distances}

We begin by giving a brief reminder on the various distances that will be used throughout this paper. 
The natural distance to study the transport of the measure $\mu_t$ by the non-local vector field $V[\mu_t]$ is the $p$-Wasserstein distance $W_p$, defined for probability measures with bounded $p$-moment $\P_p(\R^d)$ (see \cite{Villani08}):
\[
\forall \mu,\nu \in \P_p(\R^d), \quad W_p(\mu,\nu) := \left(\inf\limits_{\pi\in\Pi(\mu,\nu)}\int_{\R^d\times\R^d} |x-y|^p d\pi(x,y)\right)^{1/p},
\]
where $\Pi$ is the set of transference plans with marginals $\mu$ and $\nu$, defined  by
\[
\Pi(\mu,\nu) = \{\pi\in\P(\R^d\times\R^d)\,;\, \forall A,B\in\B(\R^d), \, \pi(A\times\R^d) = \mu(A) \text{ and } \pi(\R^d\times B) = \nu(B)\}.
\]
In the particular case $p=1$, there is an equivalent definition of $W_1$ by the Kantorovich-Rubinstein duality : 
\[
\forall \mu,\nu \in \P_p(\R^d), \quad W_1(\mu,\nu)=\sup\left\{ \int_{\R^d} f(x) d(\mu(x)-\nu(x)); \quad f\in C^{0,\Lip}_c(\R^d), \|f\|_{\Lip}\leq 1\right\}.
\] 

The Wasserstein distance was extended in \cite{PiccoliRossi14, PiccoliRossi16} to the set of positive Radon measures with possibly different masses. For $a,b>0$, the generalized Wasserstein distance $W_p^{a,b}$ is defined by: 
\[
\forall \mu,\nu\in \mathcal{M}_p(\R^d), \quad W_p^{a,b}(\mu,\nu) = \left(\inf_{\tilde \mu, \tilde \nu \in M(\R^d), |\tilde \mu| = |\tilde\nu|} 
a^p(|\mu-\tilde \mu| + |\nu-\tilde \nu|)^p + b^p W_p^p(\tilde\mu,\tilde\nu)\right)^{1/p}
\]
where $\mathcal{M}_p(\R^d)$ denotes the set of positive Radon measures with bounded $p$-moment.
\begin{rem}\label{rem:noneq}
Observe that the classical and the generalized Wasserstein distances do not generally coincide on the set of probability measures. 
Indeed, the Wasserstein distance between $\mu$ and $\nu$ represents the cost of transporting $\mu$ to $\nu$, and is inextricably linked to the distance between their supports. The generalized Wasserstein distance, on the other hand, allows one to choose between \emph{transporting} $\mu$ to $\nu$ (with a cost proportional to $b$) and \emph{creating or removing mass} from $\mu$ or $\nu$ (with a cost proportional to $a$). Taking for instance $\mu=\delta_{x_1}$ and $\nu = \delta_{x_2}$, the Wasserstein distance $W_p(\delta_{x_1},\delta_{x_2})=d(x_1,x_2)$ increases linearly with the distance between the centers of mass of $\mu$ and $\nu$.
However, one can easily see that 
\[
\Woo(\delta_{x_1},\delta_{x_2}) = \inf_{0\leq \varepsilon\leq 1} 
(|\delta_{x_1}-\varepsilon\delta_{x_1}| + |\delta_{x_2}-\varepsilon\delta_{x_2}| + \varepsilon W_p(\delta_{x_1},\delta_{x_2})) = \inf_{0\leq \varepsilon\leq 1} 
(2(1-\varepsilon) + \varepsilon d(x_1,x_2))
\]
from which it holds $\Woo(\delta_{x_1},\delta_{x_2}) = \min(d(x_1,x_2),2)$.
\end{rem}

More generally, if $\mu,\nu\in\P_p(\R^d)$, taking $\tilde{\mu}=\mu$ and $\tilde{\nu}=\nu$ in the definition of $W_p^{a,b}$ yields
\[
\quad W_p^{a,b}(\mu,\nu) \leq b W_p(\mu,\nu).
\]
On the other hand, taking $\tilde{\mu}=\tilde{\nu}=0$ yields 
\[
\quad W_p^{a,b}(\mu,\nu) \leq a (|\mu|+|\nu|).
\]
In particular, for $a=b=1$, the generalized Wasserstein distance $\Woo$ also satisfies a duality property and coincides with the \emph{Bounded Lipschitz Distance} $\rho(\mu,\nu)$ (see \cite{Dudley02}): for all $\mu,\nu\in \MR,$
\[
\Woo(\mu,\nu) = \rho(\mu,\nu) := \sup\left\{ \int_{\R^d} f(x) d(\mu(x)-\nu(x)); \; f\in C^{0,\Lip}_c(\R^d), \|f\|_{\Lip}\leq 1, \|f\|_{L^\infty}\leq 1\right\}.
\]

In turn, this Generalized Wasserstein distance was extended in \cite{PiccoliRossiTournus20} to the space $\mathcal{M}^s_1(\R^d)$ of signed measures with finite mass and bounded first moment as follows:
\[
\forall \mu, \nu\in \mathcal{M}^s_1(\R^d), \quad \mathbb{W}^{a,b}_1(\mu,\nu)  = W^{a,b}_1(\mu_+ + \nu_-, \mu_- +\nu_+ )  
\]
where $\mu_+, \mu_-, \nu_+ $ and $\nu_-$ are any measures in $\mathcal{M}(\R^d)$ such that $\mu = \mu_+ - \mu_-$ and $\nu = \nu_+ - \nu_-$.
We draw attention to the fact that for positive measures, the two generalized Wasserstein distances coincide: 
\[
\forall \mu, \nu\in \mathcal{M}_1(\R^d), \quad \mathbb{W}^{a,b}_1(\mu,\nu) = W^{a,b}_1(\mu,\nu).
\]

Again, for $a=b=1$, the duality formula holds and the Generalized Wasserstein distance $\WWoo$ is equal to the Bounded Lipschitz distance $\rho$: 
\[
\forall \mu, \nu\in \mathcal{M}^s_1(\R^d), \quad \mathbb{W}^{1,1}_1(\mu,\nu)  = \rho(\mu,\nu).
\]

From here onward, we will denote by $\rho(\mu,\nu)$ the Bounded Lipschitz distance, equal to the generalized Wasserstein distances $\Woo$ on $\MR$ and $\WWoo$ on $\MsR$. 
The properties of the Generalized Wasserstein distance mentioned above give us the following estimate, that will prove useful later on: 
\begin{equation}\label{eq:estirho}
\rho(\mu, \nu)\leq |\mu|+|\nu|.
\end{equation}

We recall other properties of the Generalized Wasserstein distance proven in \cite{PiccoliRossiTournus20} (Lemma 18 and Lemma 33). Although they hold for any $\mathbb{W}_1^{a,b}$, we write them here in the particular case $\WWoo=\rho$:

\begin{prop}\label{prop:sumproperties}
Let $\mu_1$, $\mu_2$, $\nu_1$, $\nu_2$ in $\MsR$ with finite mass on $\R^d$. The following properties hold: 
\begin{itemize}
\item $\rho(\mu_1 + \nu_1,\mu_2 + \nu_1) = \rho(\mu_1 ,\mu_2 )$
\item $\rho(\mu_1 + \nu_1,\mu_2 + \nu_2) \leq \rho(\mu_1 ,\mu_2 ) + \rho(\nu_1 ,\nu_2 )$
\end{itemize} 
\end{prop}
The following proposition, proven in \cite{PiccoliRossiTournus20}, holds for any $\mathbb{W}_1^{a,b}$. Again, for simplicity, we state it for the particular case of the distance $\rho$.
Note that to simplify notations and to differentiate from function norms, all vector norms for elements of $\R^d$, $d\geq 1$, will be written $|\cdot|$. The difference with the mass or total variation of a measure will be clear from context.
\begin{prop}\label{prop:flowproperties}
Let $v_1,v_2\in C([0,T]\times\R^d)$ be two vector fields, both satisfying for all $t\in [0,T]$ and $x,y\in\R^d$ the properties
$$
|v_i(t,x)-v_i(t,y)|\leq L |x-y|, \quad |v_i(t,x)|\leq M
$$
where $i\in\{1,2\}$. Let $\mu,\nu\in\MsR$. Let $\Phi_t^{v_i}$ denote the flow of $v_i$, that is the unique solution to 
\[
\frac{d}{dt}\Phi_t^{v_i}(x) = v_i(t,\Phi_t^{v_i}(x)); \qquad \Phi_0^{v_i}(x) = x.
\]
Then
\begin{itemize}
\item $\rho(\Phi_t^{v_1}\# \mu,\Phi_t^{v_1}\# \nu) \leq e^{Lt} \rho(\mu,\nu)$
\item $ \rho(\mu,\Phi_t^{v_1}\#\mu) \leq t M |\mu|$
\item $ \rho(\Phi_t^{v_1}\# \mu,\Phi_t^{v_2}\# \mu) \leq |\mu| \frac{e^{Lt}-1}{L} \|v_1-v_2\|_{L^\infty(0,T;\C^0)}$
\item $ \rho(\Phi_t^{v_1}\# \mu,\Phi_t^{v_2}\# \nu) \leq e^{Lt} \rho(\mu,\nu) + \min\{|\mu|,|\nu|\} \frac{e^{Lt}-1}{L} \|v_1-v_2\|_{L^\infty(0,T;\C^0)}$.
\end{itemize}
\end{prop}
The notation $\#$ used above denotes the push-forward, defined as follows: for $\mu\in\MsR$ and $\phi:\R^d\rightarrow\R^d$ a Borel map, the push-forward $\phi\#\mu$ is the measure on $\R^d$ defined by $\phi\#\mu(E):=\mu(\phi^{-1}(E))$, for any Borel set $E\subset\R^d$.


We end this section with a result of completeness that will prove central in the subsequent sections.
As remarked in \cite{PiccoliRossiTournus20}, $(\MsR,\mathbb{W}^{ab}_p)$ is not a Banach space. However, $(\MR,\Wabp)$ is (as shown in \cite{PiccoliRossi16}), and we can also show the following: 

\begin{prop}\label{prop:completeness}
$\PR$ is complete with respect to the Generalized Wasserstein distance $\Wabp$.
\end{prop}

\begin{proof}
Let $\{\mun\}\subset\PR$ be a Cauchy sequence with respect to $\Wabp$. It was proven in the proof of Prop. 4 in \cite{PiccoliRossi16} that $\{\mun\}$ is tight.
From Prokhorov's theorem, there exists $\mus\in\PR$ and a subsequence $\{\muunk\}$ of $\{\mun\}$ such that $\muunk\rightharpoonup_{k\rightarrow\infty } \mus$. 
From Theorem 3 of \cite{PiccoliRossi16}, this implies that $\Wabp(\muunk,\mus) \rightarrow 0$.
From the Cauchy property of $\{\mun\}$ and the triangular inequality, this in turn implies that $\Wabp(\mun,\mus)\rightarrow 0$.
\end{proof}

In particular, note that  $\PR$ is also complete with respect to the Bounded Lipschitz distance $\rho$.

\subsection{Comparison between the distances}\label{Sec:Compa}

From the definition of the Bounded-Lipschitz distance as a particular case of the Generalized Wasserstein distance $\WWoo$ (for $a=b=1$), we have the following property: 
\begin{equation}\label{eq:BLleqWass}
\forall \mu,\nu\in \PR, \qquad \rho(\mu,\nu)\leq W_1(\mu,\nu).
\end{equation} 

As pointed out in Remark \ref{rem:noneq}, the converse is not true in general.
However, we can show that for measures with bounded support, one can indeed control the $1-$Wasserstein distance with the Bounded Lipschitz one.

\begin{prop}\label{Prop:WassleqBL}
Let $R>0$. For all $\mu,\nu\in \PcR$, if $\supp(\mu)\cup\supp(\nu)\subset B(0,R)$, 
it holds
\[
\rho(\mu,\nu)\leq W_1(\mu,\nu) \leq C_R \rho(\mu,\nu)
\]
where $C_R=\max(1,R)$.
\end{prop}

\begin{proof}
Let $\mu,\nu\in \PcR$, such that $\supp(\mu)\cup\supp(\nu)\subset B(0,R)$. Let
\[
A:=\left\{ \int_{\R^d} f d(\mu-\nu); \; f\in C^{0,\Lip}_c(\R^d), \|f\|_{\Lip}\leq 1, \|f\|_{L^\infty}\leq 1\right\}
\]
\[ \text{and }\quad B:=\left\{ \int_{\R^d}f d(\mu-\nu); \; f\in C^{0,\Lip}_c(\R^d), \|f\|_{\Lip}\leq 1\right\}.
\]
Then $\rho(\mu,\nu) = \sup_{A} a$ and $W(\mu,\nu) = \sup_{B} b$.
It is clear that $A\subset B$, which proves the first inequality.\\
Let 
\[ \tilde B=\left\{ \int_{\R^d}f d(\mu-\nu); \; f\in C^{0,\Lip}_c(\R^d), \|f\|_{\Lip}\leq 1, f(0)=0\right\}.
\]
Clearly, $\tilde B \subset B$. Let us show that $B\subset \tilde{B}$.
Let $b\in B$. There exists $f_b\in C^{0,\Lip}_c(\R^d)$ such that $\|f_b\|_{\Lip}\leq 1$ and $b=\int_{\R^d}f_b d(\mu-\nu)$.
Let us define $\tilde{f_b}\in C(\R^d)$ such that for all $x\in B(0,R)$, $\tilde{f_b}(x) = f_b(x)-f_b(0)$. It holds $\|\tilde{f_b}\|_{\Lip(B(0,R))}\leq 1$. We prolong $\tilde{f_b}$ outside of $B(0,R)$ in such a way that $\tilde{f_b}\in C^{0,\Lip}_c(\R^d)$ $\arg\max(\tilde{f_b})\in B(0,R)$ and $\|\tilde{f_b}\|_{\Lip(\R^d)}\leq 1$.
Then since the supports of $\mu$ and $\nu$ are contained in $B(0,R)$, 
\[
\int_{\R^d}\tilde{f_b} \,d(\mu-\nu) = \int_{B(0,R)}\tilde{f_b}\, d(\mu-\nu) = \int_{B(0,R)}f_b\, d(\mu-\nu) - f(0) \int_{B(0,R)} d(\mu-\nu) = b
\]
where the last equality comes from the fact that $\mu(B(0,R))=\nu(B(0,R))=1$. This proves that $b\in \tilde{B}$, and so $B=\tilde{B}$.

Let us now show that there exists $a\in A$ such that $b\leq \max(1,R) a$.\\
If $\|\tilde{f_b}\|_{L^\infty(\R^d)}\leq 1$, then $b\in A$. \\
If $\|\tilde{f_b}\|_{L^\infty(\R^d)} > 1$, let $f_a := \tilde{f_b}/\|\tilde{f_b}\|_{L^\infty(\R^d)}$. It holds $\|f_a\|_{L^\infty(\R^d)}\leq 1$ and $\|f_a\|_{\Lip}\leq 1$. Thus $a:=\int_{\R^d} f_a \, d(\mu-\nu)  \in A$ and it holds
\[
b = \|\tilde{f_b}\|_{L^\infty(\R^d)}\int_{\R^d} \tilde{f_b}/\|\tilde{f_b}\|_{L^\infty(\R^d)} d(\mu-\nu) \leq  \|\tilde{f_b}\|_{L^\infty(\R^d)} a.
\]
Since $\tilde{f_b}(0) = 0$ and $\|\tilde{f_b}\|_{\Lip}\leq 1$, it holds $\|\tilde{f_b}\|_{L^\infty(B(0,R))}\leq R$, hence $\|\tilde{f_b}\|_{L^\infty(\R^d)}\leq R$.
We then have: 
\[
\forall b\in B, \quad \exists a\in A \quad \text{ s.t. } \quad b\leq \max(1,R) a,
\]
which implies that $\sup_B b\leq \max(1,R) \sup_A a$.
\end{proof}

It is a well-known property of the Wasserstein distances that for all $m\leq p$, for all $\mu,\nu\in\P_p(\R^d)$, 
\begin{equation}\label{eq:WasspleqWassq}
W_m(\mu,\nu) \leq W_p(\mu,\nu).
\end{equation}
The proof of this result is a simple application of the Jensen inequality \cite{Villani08}.

The converse is false in general. However, once again, we can prove more for measures with compact support in the case $m=1$.
\begin{prop}\label{Prop:WassqleqWass1}
Let $R>0$ and $p\in\N^*$. For all $\mu,\nu\in \PcR$, if $\supp(\mu)\cup\supp(\nu)\subset B(0,R)$,
\begin{equation*}
W_p(\mu,\nu) \leq (2R)^{\frac{p-1}{p}} W_1(\mu,\nu)^{\frac{1}{p}}.
\end{equation*}
\end{prop}
\begin{proof}
Let $\pi\in\Pi(\mu,\nu)$ be a transference plan with marginals $\mu$ and $\nu$. Since the supports of $\mu$ and $\nu$ are contained in $B(0,R)$, the support of $\pi$ is contained in $B(0,R)\times B(0,R)$. We can then write:
\[
\int_{\R^d\times\R^d} d(x,y)^p d\pi(x,y) = \int_{B(0,R)^2} d(x,y)^p d\pi(x,y)  \leq (2R)^{p-1} \int_{B(0,R)^2} d(x,y) d\pi(x,y)
\]
from which we deduce the claimed property.
\end{proof}

\section{Macroscopic Model}

In this section, we give a meaning to the non-linear and non-local transport equation with source: 
\begin{equation}\label{eq:transportsource}
\partial_t \mt(x) + \nabla\cdot \left(V[\mu_t](x) \mt(x)\right) = h[\mt](x), \qquad \mu_{t=0} = \muo,
\end{equation}
where the non-local vector field $V$ and source term $h$ are defined as follows:
\begin{itemize}
\item 
Let $\ba\in \Lip(\R^d;\R^d)$ satisfy Hyp. \ref{hyp:abar}.
The vector field $V\in C^{0,\Lip}(\MR; C^{0,\Lip}(\R^d))$ is defined by:
\begin{equation}\label{eq:Vdef}
\forall \mu\in \MR, \quad \forall x\in\R^d, \quad V[\mu](x) = \int_{\R^d} \ba(x-y) d\mu(y).
\end{equation}
\item Let $S\in C^0((\R^d)^{q+1}; \R)$ satisfy Hyp. \ref{hyp:S}.
The source term $h\in C^{0,\Lip}(\MR; \MsR)$ is then defined as:
\begin{equation}\label{eq:hdef}
\forall \mu\in \MR, \quad \forall x\in\R^d, \quad h[\mu](x) = \left(\int_{(\R^d)^q} S(x,y_1,\cdots,y_q) d\mu(y_1)\cdots d\mu(y_q)\right) \mu(x).
\end{equation}
\end{itemize}

The solution to \eqref{eq:transportsource} will be understood in the following weak sense:
\begin{definition}\label{Def:weaksol}
A measure-valued weak solution to \eqref{eq:transportsource} is a map $\mu\in C^0([0,T],\MsR)$ such that $\mu_{t=0} = \muo$ and for all $f\in C_c^\infty(\R^d)$,
\begin{equation}\label{eq:transportsourceweak}
\frac{d}{dt}\int_{\R^d} f(x) d\mu_t(x) = \int_{\R^d} V[\mu_t]\cdot \nabla f(x) d\mu_t(x) + \int_{\R^d}f(x) dh[\mu_t](x).
\end{equation}
\end{definition}

\begin{rem}
This model is a modified version of the one proposed in \cite{PiccoliRossi18}.
The form of the source term \eqref{eq:hdef} is slightly more general than the one of \cite{PiccoliRossi18} (where $h$ was defined as $h[\mu](x) = \left( S_1 + S_2\star\mu\right) \mu$). However we also introduce a more restrictive condition \eqref{eq:condS} that will force the source term to be a signed measure with zero total mass.
\end{rem}

The first aim of this paper will be to prove the following
\begin{theorem}\label{Th:existenceuniqueness}
Let $\muo\in\PcR$. There exists a unique weak solution to \eqref{eq:transportsource} in the space $C^0([0,T],\PcR)$. 
\end{theorem}

Notice that we are almost in the frameworks of \cite{PiccoliRossi14} and \cite{PiccoliRossiTournus20}.
In \cite{PiccoliRossi14}, existence and uniqueness was proven for a transport equation with source of the form \eqref{eq:transportsource}, for measures in $\MR$ and with source term $h\in C^{0,\Lip}(\MR, \MR)$. Since in our case, $h[\mu]$ is a signed measure, we cannot apply directly the theory of \cite{PiccoliRossi14}.
In \cite{PiccoliRossiTournus20}, existence and uniqueness was proven for a transport equation with source of the form \eqref{eq:transportsource}, for measures in $\MsR$ and with source term $h\in C^{0,\Lip}(\MsR, \MsR)$. 
However, as we will see in Section \ref{Sec:PropMacro}, the source term $h$ defined by \eqref{eq:hdef} does not satisfy some of the assumptions required in \cite{PiccoliRossiTournus20}, namely a global Lipschitz property and a global bound on the mass of $h[\mu]$.

\subsection{Properties of the model} \label{Sec:PropMacro}
We now prove that the vector field $V[\mu]$ satisfies Lipschitz and boundedness properties, provided that $\mutv$ is bounded.\\
First, notice that the continuity of $\phi$ implies that for all $R>0$ and $x\in\R^d$ such that $|x|\leq 2 R$, there exists $\phiR>0$ such that $|\phi(x)|\leq \phiR$. More specifically, since $\phi$ is Lipschitz, $\phiR=\phi_0 + 2 L_\phi R$, with $\phi_0:=\phi(0)$.
\begin{prop}\label{Prop:V}
The vector field $V$ defined by \eqref{eq:Vdef} satisfies the following:
\begin{itemize}
\item For all $\mu\in\MsR$ such that $\supp(\mu)\subset B(0,R)$, for all $ x\in B(0,R)$, 
$
|V[\mu](x)| \leq \phiR \, \mutv.
$
\item For all $(x,y)\in \R^{2d}$, for all $\mu\in\MsR$, 
$$
| V[\mu](x)-V[\mu](z) | \leq \La \mutv  \; |x-z|.
$$
\item For all $\mu, \nu\in\MsR$ such that $\supp(\mu)\cup\supp(\nu)\subset B(0,R)$, 
$$
\| V[\mu]-V[\nu] \|_{L^\infty(B(0,R))} \leq (\La+\phi_R) \, \rho(\mu,\nu).
$$
\end{itemize}
\end{prop}

\begin{proof}
Let $\mu\in\MsR$.
If $\supp(\mu)\in B(0,R)$, for all $x\in B(0,R)$,
$$
|V[\mu](x)|  =\left| \int_{B(0,R)} \phi(y-x) d\mu(y)\right| \leq \left(\sup_{(x,y)\in B(0,R)}|\phi(y-x)|\right)\;\mutv \leq \phiR \;\mutv.
$$
Secondly, for all $(x,z)\in \R^{2d}$,
\begin{equation*}
\begin{split}
|V[\mu](x) - V[\mu](z)| & = \left| \int_{\R^d} (\ba(y-x)-\ba(y-z)) d\mu(y)\right| \leq  \int_{\R^d} |\ba(y-x)-\ba(y-z)| d|\mu|(y)\\ 
& \leq 
\La |x-z| \mutv.
\end{split}
\end{equation*}
Lastly, for all $\mu, \nu\in\MsR$ such that $\supp(\mu)\cup\supp(\nu)\subset B(0,R)$ for all $x\in B(0,R)$,
\begin{equation*}
\begin{split}
| V[\mu](x)-V[\nu](x) |& = \int_{B(0,R)} \ba(y-x) d(\mu(y)-\nu(y)) \\
& \leq (L_\phi+\phi_R) \sup_{f\in \C_c^{0,\Lip}, \|f\|_\Lip\leq 1, \|f\|_\infty\leq 1} \int_{\R^d} f(y)\; d(\mu(y)-\nu(y))\\
& \leq (\La+\phi_R) \, \rho(\mu,\nu)
\end{split}
\end{equation*}
where we used the fact that for all $x\in B(0,R)$, the function $y\mapsto (\La+\phi_R)^{-1} \ba(y-x)$ has both Lipschitz and $L^\infty$ norms bounded by 1, and the definition of $\rho$.
\end{proof}

\begin{prop}\label{Prop:h}
The source term $h$ defined by \eqref{eq:hdef} satisfies the following:
\begin{enumerate}[(i)]
\item $\forall\mu\in\MsR$, $h[\mu](\R^d) = 0$
\item $\forall\mu\in\MsR$, $\supp(h[\mu]) = \supp(\mu)$
\item There exists $L_h$ such that for all $\mu,\nu\in\MsR$ with compact support and with bounded total variation $\mutv\leq Q$ and $\TV{\nu}\leq Q$,
$$
\rho(h[\mu],h[\nu]) \leq L_h \rho(\mu,\nu).
$$
\item $\forall\mu\in\MR$, $\TV{h[\mu]}\leq \bS \; |\mu|^{q+1}$.
\item $\forall \mu \in \MR$, $\forall E\subset \R^d$,  $h[E] \geq -\bS \; |\mu| \; \mu(E)$.
\end{enumerate}
\end{prop}

\begin{proof}
Let $\mu\in\MsR$. From the definition of $h$, we compute: 
\begin{equation*}
\begin{split}
h[\mu](\R^d) = & \int_{(\R^d)^{q+1}}  S(y_0,y_1,\cdots,y_q) d\mu(y_0)d\mu(y_1)\cdots d\mu(y_q)\\
= & \frac{1}{2}\int_{(\R^d)^{q+1}} S(y_0,\cdots, y_i,\cdots,y_j,\cdots,y_q) d\mu(y_0)\cdots d\mu(y_q) \\
& + \frac{1}{2}\int_{(\R^d)^{q+1}} S(y_0,\cdots, y_j,\cdots,y_i,\cdots,y_q) d\mu(y_0)\cdots d\mu(y_q) 
\end{split}
\end{equation*}
where we used the change of variables $y_i \leftrightarrow y_j$ to write the second term. Then, using the skew-symmetric property \eqref{eq:condS}, we obtain $h[\mu](\R^d)=0$.

The second property is immediate from the definition of $h[\mu]$. 

For the third point, 
let $\mu,\nu\in\MsR$ with compact support, and satisfying $\mutv\leq Q$ and $\TV{\nu}\leq Q$.
For all $f\in \C_c^{0,\Lip}$ such that $\|f\|_\infty\leq 1$ and $\|f\|_\Lip\leq 1$,
%
\begin{equation*}
\begin{split}
\int_{\R^d} f(x)\; &d(h[\mu]-h[\nu]) = 
\int_{\R^d} f(x)\; \int_{\R^{qd}}  S(x,y_1\cdots,y_q) d\mu(y_1)\cdots d\mu(y_q) d\mu(x) \\
& -  \int_{\R^d} f(x)\; \int_{\R^{qd}} S(x,y_1\cdots,y_q) d\nu(y_1)\cdots d\nu(y_q) d\nu(x) \\
= & \int_{\R^d} f(x)\; \int_{\R^{qd}}  S(x,y_1\cdots,y_q) d\mu(y_1)\cdots d\mu(y_q) d(\mu(x)-\nu(x)) \\
& + \sum_{i=1}^q \int_{\R^d} f(x)\; \int_{\R^{qd}}  S(x,y_1\cdots,y_q) d\mu(y_1)\cdots d\mu(y_i) d\nu(y_{i+1})\cdots d\nu(y_q) d\nu(x) \\
& - \sum_{i=1}^q \int_{\R^d} f(x)\; \int_{\R^{qd}}  S(x,y_1\cdots,y_q) d\mu(y_1)\cdots d\mu(y_{i-1}) d\nu(y_{i})\cdots d\nu(y_q) d\nu(x) \\
= & \int_{\R^d} f(x)\; \int_{\R^{qd}}  S(x,y_1\cdots,y_q) d\mu(y_1)\cdots d\mu(y_q) d(\mu(x)-\nu(x)) \\
& + \sum_{i=1}^q \int_{\R^d} f(x)\; \int_{\R^{qd}}  S(x,y_1\cdots,y_q) d\mu(y_1)\cdots d(\mu(y_i)-d\nu(y_{i})) d\nu(y_{i+1})\cdots d\nu(y_q) d\nu(x) \\
:= & \; A(f) + \sum_{i=1}^q B_i(f).
\end{split}
\end{equation*}
We begin by studying the first term $A(f):= \int_{\R^d} f(x)\; \psi(x) d(\mu(x)-\nu(x))$, where $\psi$ is defined by $\psi : x \mapsto  \int_{\R^{qd}}  S(x,y_1\cdots,y_q) d\mu(y_1)\cdots d\mu(y_q)$.
Notice that 
$$|\psi(x)| =  \left|\int_{\R^{qd}}  S(x,y_1\cdots,y_q) d\mu(y_1)\cdots d\mu(y_q)\right| \leq \bS \mutv^q \leq \bS Q^q.$$
Furthermore, for all $(x,z)\in\R^{2d}$, 
$$
|\psi(x)-\psi(z)| = \left|\int_{\R^{qd}}  (S(x,y_1\cdots,y_q)-S(z,y_1\cdots,y_q)) d\mu(y_1)\cdots d\mu(y_q)\right| \leq L_S \mutv^q |x-z|
$$
where we used the Lipschitz property of $S$ \eqref{eq:psiSklip}.
Thus, the function $x\mapsto f(x)\psi(x)$ satisfies: 
\begin{equation*}
\forall x\in\R^d, \quad |f(x)\psi(x)|\leq \bS Q^q.
\end{equation*}
It also satisfies:  for all $(x,z)\in\R^{2d}$,
\begin{equation*}
|f(x)\psi(x)-f(z)\psi(z)| = |f(x)(\psi(x)-\psi(z))+(f(x)-f(z))\psi(z)| \leq (L_S+\bS) Q^q |x-z|.
\end{equation*}
This implies that the function $g:x\mapsto \frac{1}{Q^q(\bS+L_S)}f(x)\psi(x)$ satisfies $g\in\C_c^{0,\Lip}$,
$\|g\|_\infty\leq 1$ and $\|g\|_\Lip\leq 1$.
Then, using the defintion of $\rho$, 
we deduce that
\begin{equation*}
\begin{split}
A(f) & = Q^q(L_S+\bS) \int_{\R^d}g(x) d(\mu(x)-\nu(x))\\
& \leq Q^q(L_S+\bS) \sup_{f\in \C_c^{0,\Lip}, \|f\|_\infty\leq 1, \|f\|_\Lip\leq 1} \int_{\R^d}f(x) d(\mu(x)-\nu(x)) 
= Q^q(L_S+\bS)\, \rho(\mu,\nu).
\end{split}
\end{equation*}

Now, let
$\zeta_i : y_i\mapsto \int_{\R^{qd}} f(x)\; S(x,y_1\cdots,y_q) d\mu(y_1)\cdots d\mu(y_{i-1}) d\nu(y_{i+1})\cdots d\nu(y_q) d\nu(x)$
and $B_i(f) := \int_{\R^d} \zeta_i(y_i) d(\mu(y_i)-d\nu(y_{i}))$.
It holds:
$$
\forall y_i\in\R^d, \quad |\zeta_i(y_i)| \leq \|f\|_{L^\infty} \|S\|_{L^\infty} \mutv^{i-1} \TV{\nu}^{q-i+1} \leq \bS Q^q.
$$
Moreover, for all $(y_i,z_i)\in\R^{2d}$,
\begin{equation*}
\begin{split}
& |\zeta_i(y_i)-\zeta_i(z_i)| \\
 = &\left| \int_{\R^{qd}} f(x)\;  (S(x,y_1,\cdots ,y_q)-S(x,y_1,\cdots ,z_i,\cdots, y_q)) d\mu(y_1)\cdots d\mu(y_{i-1}) d\nu(y_{i+1})\cdots d\nu(y_q) d\nu(x)\right| \\
\leq & \|f\|_{L^\infty} L_S |y_i-z_i| \mutv^{i-1} \TV{\nu}^{q-i+1}  \leq L_S Q^q |y_i-z_i|.
\end{split}
\end{equation*}
Hence, the function $g_i:y_i\mapsto \frac{1}{Q^q(L_s+\bS)}\zeta_i(y_i)$ satisfies $g_i\in\C^{0,\Lip}$,
$\|g_i\|_\infty\leq 1$ and $\|g_i\|_\Lip\leq 1$
\begin{equation*}
\begin{split}
B_i(f) 
\leq Q^q(L_S+\bS) \sup_{f\in \C_c^{0,\Lip}, \|f\|_\infty\leq 1, \|f\|_\Lip\leq 1} \int_{\R^d}f(x) d(\mu(x)-\nu(x)) = Q^q(L_S+\bS) \rho(\mu,\nu),
\end{split}
\end{equation*}
We conclude that for all $f\in \C_c^{0,\Lip}$ such that $\|f\|_\infty\leq 1$ and $\|f\|_\Lip\leq 1$,
\begin{equation*}
\begin{split}
\int_{\R^d} f(x) d(h[\mu](x)-h[\nu](x) ) = A(f) + \sum_{i=1}^q B_i(f) \leq (q+1)Q^q(L_S+\bS) \rho(\mu,\nu),
\end{split}
\end{equation*}
which implies
the desired property.

For the fourth point, let $\mu\in\MsR$. From the definition of $h$ follows immediately: $\TV{h[\mu]} \leq \bS \;|\mu|^{q+1}$.


Lastly, for all $\mu\in\MR$ and $E\subset \R^d$,
$$h[\mu](E) = \int_{E} \int_{\R^{dq}} S(x, y_1,\cdots, y_q) d\mu(y_1)\cdots d\mu(y_q) d\mu(x) \geq - \bS |\mu|^q \mu(E).$$

\end{proof}

\subsection{Numerical Scheme}\label{Sec:numscheme}

In \cite{PiccoliRossiTournus20}, existence of the solution to \eqref{eq:transportsource} was proven by showing that it is the limit of a numerical scheme discretizing time.
It would seem natural to apply directly the results of \cite{PiccoliRossiTournus20} on well-posedness of the equation in $\MsR$. However,  the conditions on the source function $h$ required in \cite{PiccoliRossiTournus20}, namely
\begin{equation}\label{eq:proph}
\|h[\mu]-h[\nu]\| \leq L_h \|\mu - \nu\|, \quad |h[\mu]|\leq P \quad \text{ and } \supp(h[\mu])\subset B_0(R)
\end{equation}
uniformly for all $\mu, \nu \in \MsR$ 
are not satisfied in our setting (since $L_h$ and $P$ depend on $|\mu|$, $|\nu|$, as seen in Proposition \ref{Prop:h}).
Instead, we notice that they do hold uniformly for $\mu, \nu \in \PcR$. 
Hence if the numerical scheme designed in \cite{PiccoliRossiTournus20} preserved mass and positivity, one could hope to adapt the proof by restricting it to probability measures. However, we can show that the scheme of \cite{PiccoliRossiTournus20} preserves neither positivity, nor total variation (see Appendix \ref{Sec:SchemePRT}).

For this reason, in order to prove existence of the solution to \eqref{eq:transportsource}, we design a new operator-splitting numerical scheme that conserves mass and positivity (hence total variation). The inequalities \eqref{eq:proph} will then hold for all solutions of the scheme, which will allow us to prove that it converges (with a technique very close to the techniques of \cite{PiccoliRossi14, PiccoliRossiTournus20}) in the space $C([0,T]),\PR)$ (Section \ref{Sec:numscheme}). It will only remain to prove that the limit of the scheme $\bmu$ is indeed a solution to \eqref{eq:transportsource}, and that this solution is unique (Section \ref{Sec:ExUniq}). 

%

\paragraph{Numerical Scheme $\S$.}
Let $T>0$, $k\in\N$, and let $\dt = \frac{T}{2^k}$. 
Set $\muo^k := \muo$.\\
For all $n\in\N$, 
\begin{itemize}
\item $\forall t\in (n\dt,(n+\frac{1}{2})\dt]$, let $\tau = t-n\dt$, and: 
$$\mutk := \mu_{n\dt}^k+ 2\tau h[\mu_{n\dt}^k].$$
\item $\forall t\in ((n+\frac{1}{2})\dt,(n+1)\dt]$, let $\tau' = t-(n+\frac{1}{2})\dt$, and: 
$$\mutk := \Phi_{2\tau'}^{V[\mu_{n\dt}^k]}\#\mu_{(n+\frac12)\dt}^k.$$
\end{itemize}
A schematic illustration of the scheme $\S$ is provided in Fig. \ref{Fig:Scheme}.
\begin{figure}
\centering
\includegraphics[scale=1.2]{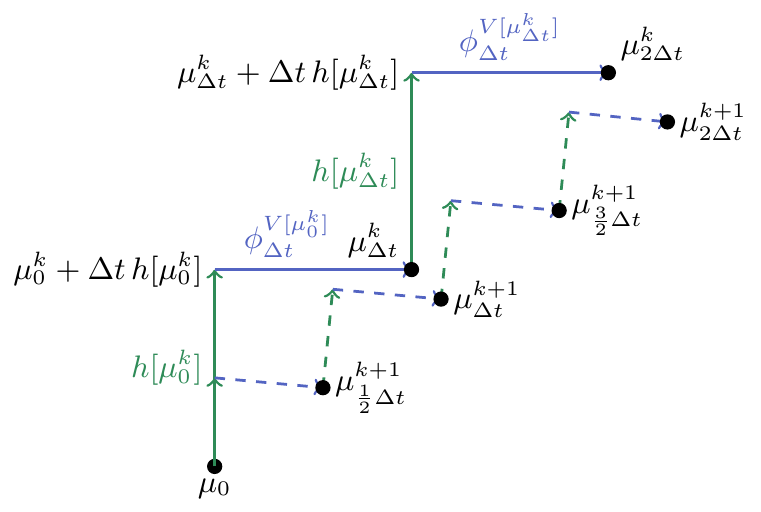}
\caption{Illustration of two steps $k$ (full lines) and $k+1$ (dashed lines) of the operator-splitting numerical scheme $\S$. The source and transport operators are respectively represented by green and blue arrows.}\label{Fig:Scheme}
\end{figure}
As stated above, we begin by proving a key property of the scheme $\S$: it preserves mass and positivity.

\begin{prop}\label{Prop:muproba}
If $\muo\in \PR$, then for all $k\geq \log_2(\bS T)$, for all $t\in [0,T]$, $\mutk \in \PR$.
\end{prop}

\begin{proof}
Let $\muo\in\PR$.
We first show that $\mutk(\R^d)=1$ for all $k\in\N$ and $t\in [0,T]$.
Suppose that for some $n\in\N$, $\munk(\R^d)=1$.
\begin{itemize}
\item For all $t\in (n\dt,(n+\frac{1}{2})\dt]$, from Proposition \ref{Prop:h},
$$
\mutk(\R^d) = \munk(\R^d)+ 2(t-n\dt) h[\mu_{n\dt}^k](\R^d) =1 + 0 =1.
$$ 
\item For all $t\in ((n+\frac{1}{2})\dt,(n+1)\dt]$,
$$
\mutk (\R^d) := (\Phi_{2(t-(n+\frac{1}{2})\dt)}^{V[\mu_{n\dt}^k]}\#\mu_{(n+\frac12)\dt}^k) (\R^d) = \mu_{(n+\frac12)\dt}^k(\Phi_{-2(t-(n+\frac{1}{2})\dt)}^{V[\mu_{n\dt}^k]}(\R^d)) = \mu_{(n+\frac12)\dt}^k(\R^d) = 1.
$$
\end{itemize}
which proves that $\mutk(\R^d)=1$ for all $t\in [0,T]$ by induction on $n$. We now show that $\mu_t^k\in\MR$ for all $k\in\N$ and $t\in [0,T]$.
Suppose that for some $n\in\N$, for all $E\subset\R^d$, $\munk(E)\geq 0$.
\begin{itemize}
\item For all $t\in (n\dt,(n+\frac{1}{2})\dt]$, for all $E\subset\R^d$, since $k\geq \log_2(\bS T)$,
\begin{equation*}
\begin{split}
\mutk(E) & = \munk(E)+ 2(t-n\dt) h[\mu_{n\dt}^k](E) \geq \munk(E) - \dt \bS \munk(\R^d)^k  \munk(E) \\
& \geq (1- 2^{-k}T \bS)  \munk(E) \geq 0 ,
\end{split}
\end{equation*}
where we used point $(v)$ of Prop. \ref{Prop:h}.
\item For all $t\in ((n+\frac{1}{2})\dt,(n+1)\dt]$, for all $E\subset\R^d$,
$$
\mutk (E) := (\Phi_{2(t-(n+\frac{1}{2})\dt)}^{V[\mu_{n\dt}^k]}\#\mu_{(n+\frac12)\dt}^k) (E)
=\mu_{(n+\frac12)\dt}^k(\Phi_{-2(t-(n+\frac{1}{2})\dt)}^{V[\mu_{n\dt}^k]}(E)) \geq 0
$$
by definition of the push-forward.
\end{itemize}
The result is proven by induction on $n$.
\end{proof}

We also prove another key property of the scheme: it preserves compactness of the support. 
\begin{prop}\label{Prop:suppmu}
Let $\muo\in \PcR$ and $R>0$ such that $\supp(\muo)\subset B(0,R)$. Then there exists $R_T$ independent of $k$ such that for all $t\in[0,T]$, for all $k\in\N$, $\supp(\mutk)\subset B(0,R_T)$.
\end{prop}

\begin{proof}
Let $k\in\N$ and suppose that for some $n\in\N$, $\supp(\munk)\subset B(0,\Rnk)$.\\
For all $t\in (n\dt,(n+\frac{1}{2})\dt]$, $\supp(\mutk) = \supp(\munk)\cup\supp(h[\munk]) = \supp(\munk) \subset B(0,\Rnk)$ from point $(ii)$ of Prop. \ref{Prop:h}.\\
For all $t\in ((n+\frac{1}{2})\dt,(n+1)\dt]$, 
$\mutk(x) = \mu_{(n+\frac12)\dt}^k(\Phi_{-2(t-(n+\frac{1}{2})\dt)}^{V[\mu_{n\dt}^k]}(x))$, so from Proposition \ref{Prop:V}, 
$$
\supp(\mutk) \subset B(0,\Rnk + \phi_{\Rnk} \dt) = B(0,\Rnk+ (\phi_0+ 2 L_\phi\Rnk) \dt) = B(0, R_{n+1,k}), 
$$
with $R_{n+1,k}:=\phi_0\dt + \Rnk (1 + 2 L_\phi\dt) $.
By induction, one can prove that 
for $t\in [(n-1)\dt,n\dt]$, $\supp(\mutk) \subset B(0,\Rnk)$, with 
\[
\Rnk= \phi_0\dt \sum_{i=0}^n (1+2\La\dt)^i + R(1+2 L_\phi \dt )^n = (1+2 L_\phi \dt )^n(\frac{\phi_0}{2\La}+R) -\frac{\phi_0}{2\La}.
\]
Since $n\leq 2^k$, for all $n\in \{0,\cdots, 2^k\}$, $\Rnk\leq (1+2 L_\phi T 2^{-k} )^{2^k}(\frac{\phi_0}{2\La}+R) -\frac{\phi_0}{2\La}$.
Moreover, 
$$
\lim_{k\rightarrow\infty} (1+2\La T 2^{-k})^{2^k} = e^{2\La T}, 
$$
so there exists $R_T$ independent of $k$ such that for all $t\in [0,T]$, $\supp(\mutk) \subset B(0,R_T)$.
\end{proof}

Propositions \ref{Prop:muproba} and \ref{Prop:suppmu} allow us to state the main result of this section.
\begin{prop}
Given $V$, $h$ defined by \eqref{eq:Vdef} and \eqref{eq:hdef} and $\muo\in\PcR$, the sequence $\mu^k$ is a Cauchy sequence for the space $(C([0,T], \PR), \D)$, where 
$$\D(\mu, \nu) := \sup_{t\in[0,T]} \rho(\mu_t, \nu_t).$$ 
\end{prop}


\begin{proof}
Let $k,n\in\N$, with $n\leq 2^k$. Let $\dt = 2^{-k} T$.
Suppose that $\supp(\mu_0)\subset B(0,R)$.
Notice that from Propositions \ref{Prop:V}, \ref{Prop:muproba} and \ref{Prop:suppmu}, we have an $L^\infty$ bound on $V[\mutk]$ independent of $t$ and $k$: for all $x\in B(0,R_T)$, for all $t\in [0,T]$, $|V[\mutk](x)|\leq M_V:= \phi_{R_T}$. 
We also have uniform Lipschitz constants for $V[\cdot]$ and $V[\mu_t^k](\cdot)$. For all $t,s\in [0,T]$, for all $\mu_t^k$, $\mu_s^l$ solutions to $\S$ with initial data $\mu_0$, it holds
\[|V[\mu_t^k](x)-V[\mu_t^k](z)| \leq \La |x-z| \quad
\text{ and } \quad
\|V[\mu_t^k]-V[\mu_s^l]\|_{L^\infty} \leq L_V \rho(\mu_t^k,\mu_s^l)
\]
where $L_V:=L_\phi+\phi_{R_T}$.
We then estimate:
\begin{equation}\label{eq:eststep}
\begin{split}
\rho(\mu^k_{n\dt},\mu^k_{(n+1)\dt}) 
& \leq \rho(\mu^k_{n\dt},\mu^k_{(n+\frac{1}{2})\dt})+\rho(\mu^k_{(n+\frac12)\dt},\mu^k_{(n+1)\dt}) \\
& \leq \rho(\mu^k_{n\dt}, \mu_{n\dt}^k+ \dt \, h[\mu_{n\dt}^k] ) + M_V \dt,
\end{split}
\end{equation}
from Prop. \ref{prop:flowproperties}.
Notice that $\mu^k_{n\dt}\in\PcR$ and $\mu_{n\dt}^k+ \dt \, h[\mu_{n\dt}^k]\in \MsR$. 
\begin{equation*}
\begin{split}
 \rho(\mu^k_{n\dt}, \mu_{n\dt}^k + \dt \, h[\mu_{n\dt}^k] ) &  = \dt \, \rho( 0, h[\mu_{n\dt}^k] ) \leq \dt |h[\mu_{n\dt}^k]| \leq \dt \bS
\end{split}
\end{equation*}
from Equation \eqref{eq:estirho}, Prop. \ref{prop:sumproperties} and Prop.  \ref{Prop:h}.
Thus, coming back to \eqref{eq:eststep}, 
$$
\rho(\mu^k_{n\dt},\mu^k_{(n+1)\dt}) \leq \dt (\bS+M_V).
$$
It follows that for all $p\in\N$ such that $n+p\leq 2^k$, 
\begin{equation*}
\rho(\mu^k_{n\dt},\mu^k_{(n+p)\dt}) \leq p \dt (\bS+M_V).
\end{equation*}
Generalizing for all $t,s\in [0,T]$, $t<s$, there exists $n,p\in\N$ such that $t = n\dt-\tilde t$ and $s=(n+p)\dt+\tilde s$, with $\tilde t, \tilde s\in [0,\dt)$.
Then 
$$
\Woo(\mu^k_{t},\mu^k_{s}) \leq \Woo(\mu^k_{t},\mu^k_{n\dt}) + \Woo(\mu^k_{n\dt},\mu^k_{(n+p)\dt}) + \Woo(\mu^k_{(n+p)\dt},\mu^k_{s}).
$$
If $\tilde t\leq \frac12 \dt$, $\rho(\mu^k_{t},\mu^k_{n\dt})\leq \bS \tilde t$.\\
If $\tilde t\geq \frac12 \dt$, $\rho(\mu^k_{t},\mu^k_{n\dt})\leq \bS \frac{\dt}{2}+ ( \tilde t - \frac{\dt}{2}) M_V \leq  \bS \tilde t +  \tilde t  M_V$.
The same reasoning for $\tilde s$ implies
\begin{equation}\label{eq:eststepts}
\rho(\mu^k_{t},\mu^k_{s}) \leq (\bS+M_V) \tilde t + p(\bS+M_V) + (\bS+M_V) \tilde s = (\bS+M_V) (s- t).
\end{equation}

We also estimate: 
\begin{equation}\label{eq:eststepndn}
\begin{split}
\rho(\mu^{k+1}_{(n+\frac12)\dt},\mu^k_{n\dt})\leq \rho(\mu^{k+1}_{(n+\frac12)\dt},\mu^{k+1}_{n\dt})+\rho(\mu^{k+1}_{n\dt},\mu^k_{n\dt})\leq 
\frac{\dt}{2}(\bS+M_V) + \rho(\mu^{k+1}_{n\dt},\mu^k_{n\dt}).
\end{split}
\end{equation}

We now aim to estimate $\rho(\mu^{k}_{(n+1)\dt},\mu^{k+1}_{(n+1)\dt})$ as a function of $\rho(\mu^{k}_{n\dt},\mu^{k+1}_{n\dt})$.
Let $H_m^j := h[\mu^j_{m\dt}]$ and $\nu_m^j := \Phi_{\dt/2}^{V[\mu^j_{m\dt}]}$.

It holds: 
\begin{equation*}
\begin{split}
\munnk = \Phi^{V[\munk]}_{\dt} \# \left( \munk + \dt\, h[\munk] \right) 
 = \nunk\# \nunk\# \left( \munk  + \dt \Hnk \right)
\end{split}
\end{equation*}
and
\begin{equation*}
\begin{split}
\munnkk & = \Phi^{V[\mundkk]}_{\dt/2} \# \left( \mundkk + \frac{\dt}{2} h[\mundkk] \right) \\
& = \nundkk\# \left( \nunkk \# ( \munkk + \frac{\dt}{2} \Hnkk ) + \frac{\dt}{2} \Hndkk\right) .
\end{split}
\end{equation*}

Hence, 
\begin{equation*}
\begin{split}
\rho(\mu^{k}_{(n+1)\dt},&\mu^{k+1}_{(n+1)\dt})
\leq  \rho(\nunk\# \nunk\# \munk,  \nundkk\# \nunkk \#  \munkk ) \\
& + \frac{\dt}{2} \rho(\nunk\# \nunk\#\Hnk,  \nundkk\# \nunkk \# \Hnkk) 
 + \frac{\dt}{2} \rho(\nunk\# \nunk\#\Hnk, \nundk \# \Hndkk).
\end{split}
\end{equation*}

We study independently the three terms of the inequality. According to Proposition \ref{prop:flowproperties} (see also \cite{PiccoliRossi14} and \cite{PiccoliRossiTournus20}), 
\begin{equation*}
\begin{split}
 & \rho(\nunk\# \nunk\# \munk,  \nundkk\# \nunkk \#  \munkk ) \\
  \leq 
 & e^{\La \frac{\dt}{2}} \rho(\nunk\# \munk, \nunkk \#  \munkk ) + \frac{e^{\La\dtt}-1}{\La} \|V[\munk]-V[\mundkk] \|_{C^0} \\
 \leq 
 & (1+ \La \dt) \rho(\nunk\# \munk, \nunkk \#  \munkk ) + \dt \|V[\munk]-V[\mundkk] \|_{C^0}.  
\end{split}
\end{equation*}
According to Proposition \ref{Prop:V} and equation \eqref{eq:eststepndn},
$$
\|V[\munk]-V[\mundkk] \|_{C^0} \leq L_V \rho(\munk, \mundkk) \leq L_V (\dtt (\bS+M_V) + \rho(\munkk, \munk) ).
$$
Similarly, 
\begin{equation*}
\begin{split}
  \rho(\nunk\# \munk, \nunkk \#  \munkk ) & \leq (1+ \La \dt) \rho(\munk, \munkk) + \dt \|V[\munk]-V[\munkk] \|_{C^0} \\
  & \leq (1+ (\La+L_V) \dt) \rho(\munk, \munkk).
\end{split}
\end{equation*}
Thus we obtain 
\begin{equation*}
\begin{split}
 & \rho(\nunk\# \nunk\# \munk,  \nundkk\# \nunkk \#  \munkk ) \\
\leq & (1+ \La \dt)(1+ (\La+L_V) \dt) \rho(\munk, \munkk) + \dt L_V (\dtt (\bS+M_V) + \rho(\munkk, \munk) ) \\
\leq & (1+ 2( \La+L_V) \dt + \La ( \La+L_V) \dt^2) \rho(\munk, \munkk) + \frac{\La}{2} (\bS+M_V) \dt^2.
\end{split}
\end{equation*}

We treat the second term in a similar way.
\begin{equation*}
\begin{split}
 & \rho(\nunk\# \nunk\#\Hnk,  \nundkk\# \nunkk \# \Hnkk)\\
 \leq & (1+\La\dt) \rho(\nunk\#\Hnk, \nunkk \# \Hnkk) + \dt \|V[\munk]-V[\mundkk] \|_{C^0}.
\end{split}
\end{equation*}
We have: 
\begin{equation*}
\begin{split}
\rho(\nunk\#\Hnk, \nunkk \# \Hnkk) & \leq (1 + \La\dt) \rho(\Hnk, \Hnkk) + \dt  \|V[\munk]-V[\munkk] \|_{C^0} \\
& \leq (1 + \La\dt)L_h \rho(\munk, \munkk) + \dt L_V \rho(\munk, \munkk)\\
& \leq (1 + (\La L_h+L_V)\dt) \rho(\munk, \munkk) .
\end{split}
\end{equation*}
Thus, 
\begin{equation*}
\begin{split}
 & \rho(\nunk\# \nunk\#\Hnk,  \nundkk\# \nunkk \# \Hnkk)\\
 \leq & (1+\La\dt) (1 + (\La L_h+L_V)\dt) \rho(\munk, \munkk) 
  + \dt L_V [\frac{\dt}{2}(2\bS+M_V) + \rho(\mu^{k+1}_{n\dt},\mu^k_{n\dt})] \\
  \leq & (1+ (\La(L_h+1) +2 L_V)\dt + \La(\La L_h+L_V) \dt^2) \rho(\mu^{k+1}_{n\dt},\mu^k_{n\dt}) + \frac{L_V}{2} (\bS+M_V) \dt^2.
\end{split}
\end{equation*}

Lastly, for the third term we have: 
\begin{equation*}
\begin{split}
 \rho(\nunk\# & \nunk\#\Hnk, \nundk \# \Hndkk) \leq 
 (1+\La\dt) \rho( \nunk\#\Hnk, \Hndkk) + \dt \|V[\munk]-V[\mundkk] \|_{C^0}\\
 & \leq  (1+\La\dt) [\rho( \nunk\#\Hnk, \Hnk) + \rho( \Hnk, \Hndkk)] + \dt L_V \rho(\munk, \mundkk) \\
 & \leq (1+\La\dt) [\dtt M_V + L_h \rho(\munk, \mundkk)] + \dt L_V \rho(\munk, \mundkk) \\
& \leq (1+\La\dt) M_V \dtt + (1+ (\La L_h+L_V)\dt)[ \frac{\dt}{2}(\bS+M_V) + \rho(\mu^{k+1}_{n\dt},\mu^k_{n\dt})] \\
& \leq \frac{\dt}{2} (L_S+2M_V) + O(\dt^2) + (1+  (\La L_h+L_V)\dt) \rho(\mu^{k+1}_{n\dt},\mu^k_{n\dt}).
\end{split}
\end{equation*}

Gathering the three terms together, we have the following estimate: 
\begin{equation*}
\rho(\mu^{k}_{(n+1)\dt},\mu^{k+1}_{(n+1)\dt}) \leq (1+C_1\dt) \rho(\mu^{k+1}_{n\dt},\mu^k_{n\dt}) + C_2 \dt^2
\end{equation*}
where
$C_1$ and $C_2$ depend on the constants $\La$, $L_V$, $L_h$, $M_V$ and $\bS$.
Thus, by induction on $n$,
\begin{equation*}
\rho(\mu^{k}_{n\dt},\mu^{k+1}_{n\dt}) \leq C_2\dt^2 \frac{(1+C_1\dt)^n-1}{1+C_1\dt -1} \leq 2n C_2\dt.
\end{equation*}
This allows us to prove the convergence of $\mu_t^k$ for every $t\in [0,T]$. For instance, for $t=T$, i.e. $n = T/\dt$, we have 
\begin{equation*}
\rho(\mu^{k}_{T},\mu^{k+1}_{T}) \leq  2 C_2\dt = 2 T C_2 2^{-k}
\end{equation*}
and for all $l,k\in \N,$
\begin{equation*}
\rho(\mu^{k}_{T},\mu^{k+l}_{T}) \leq  2 C_2\left(\frac{1}{2^k}+\frac{1}{2^{k+1}}+\cdots +\frac{1}{2^{k+l-1}}\right)\leq \frac{4C_2}{2^k}.
\end{equation*}
A similar estimation holds for any $t\in (0,T)$ (see \cite{PiccoliRossi14}).
This proves that the sequence $\mu^k$ is a Cauchy sequence for the space $(C([0,T], \PR), \D)$.
\end{proof}

As an immediate consequence, since $(C([0,T], \PR), \D)$  is complete (see Proposition \ref{prop:completeness}), it follows that there exists $\bmu\in (C([0,T], \PR)$ such that 
\begin{equation*}
\lim_{k\rightarrow\infty} \D(\mu^k,\bmu) = 0.
\end{equation*}%


\subsection{Existence and uniqueness of the solution}\label{Sec:ExUniq}

Let $\bmu_t := \lim_{k\rightarrow\infty} \mu_t^k$ denote the limit of the sequence constructed with the numerical scheme defined in the previous section. We now prove that it is indeed a weak solution of \eqref{eq:transportsource}.
We aim to prove that for all $f\in C_c^\infty((0,T)\times \R^d)$, it holds
$$
\int_0^T \; \left( \int_{\R^d} \; (\partial_t f + V[\bmu_t]\cdot \nabla f) \; d\bmu_t + \int_{\R^d} \; f \; dh[\bmu_t] \right) dt =0.
$$

We begin by proving the following result: 
\begin{lemma}\label{lemma:convergenceweak}
Let $\mu_0\in\PcR$ and let $\mu^k\in  C([0,T], \PcR)$ denote the solution to the numerical scheme $\mathbb{S}$ with initial data $\mu_0$. Let $\dt_k:=2^{-k}T$. For all $f\in C_c^\infty((0,T)\times \R^d)$, it holds: 
$$
\lim_{k\rightarrow\infty} \sum_{n=0}^{2^k-1}\int_{n\dt_k}^{(n+1)\dt_k} \; \left( \int_{\R^d} \; (\partial_t f + V[\munk]\cdot \nabla f) \; d\mu_t^k + \int_{\R^d} \; f \; dh[\munk] \right)dt =0.
$$
\end{lemma}

\begin{proof}
Let $k\in\N$ and $\dt:=2^{-k}T$.
From the definition of the numerical scheme, we have
\begin{equation}\label{eq:sumint}
\begin{split}
& \int_{n\dt}^{(n+1)\dt} \; \left( \int_{\R^d} \; (\partial_t f + V[\munk]\cdot \nabla f) \; d\mu_t^k + \int_{\R^d} \; f \; dh[\munk] \right) dt\\
= & 
\int_{n\dt}^{(n+\frac{1}{2})\dt} \; \left( \int_{\R^d} \; (\partial_t f + V[\munk]\cdot \nabla f) \; d(\munk + 2(t-n\dt) h[\munk])  \right) dt \\
& + \int_{(n+\frac{1}{2})\dt}^{(n+1)\dt} \; \left( \int_{\R^d} \; (\partial_t f + V[\munk]\cdot \nabla f) \; d(\Phi^{V[\munk]}_{2(t-(n+\frac{1}{2}))\dt}\#\mundk) \right) dt + \int_{n\dt}^{(n+1)\dt} \; \int_{\R^d} \; f \; dh[\munk] dt \\
= & \int_{n\dt}^{(n+\frac{1}{2})\dt} \; \left( \int_{\R^d} \; \partial_t f  \; d(\munk + 2(t-n\dt) h[\munk])  \right) dt + \int_{n\dt}^{(n+1)\dt} \; \int_{\R^d} \; f \; dh[\munk] dt \\
& + \int_{n\dt}^{(n+\frac{1}{2})\dt} \; \left( \int_{\R^d} \; ( V[\munk]\cdot \nabla f) \; d(\munk + 2(t-n\dt) h[\munk])  \right) dt \\
& + \int_{(n+\frac{1}{2}\dt}^{(n+1)\dt} \; \left( \int_{\R^d} \; (\partial_t f + V[\munk]\cdot \nabla f) \; d(\Phi^{V[\munk]}_{2(t-(n+\frac{1}{2}))\dt}\#\mundk) \right) dt.
\end{split}
\end{equation}

We begin by noticing that $\munk + 2(t-n\dt) h[\munk]$ is a weak solution on $(n\dt, (n+\frac12)\dt)$ to 
$$
\partial_t\nu_t = 2h[\munk], \quad \nu_{n\dt} = \munk, 
$$
so it satisfies: 
\begin{equation}\label{eq:sumint1}
\begin{split}
&\int_{n\dt}^{(n+\frac{1}{2})\dt} \;  \int_{\R^d} \; \partial_t f  \; d(\munk + 2(t-n\dt) h[\munk])   dt \\
= & -2 \int_{n\dt}^{(n+\frac{1}{2})\dt} \;  \int_{\R^d} \; f  \; dh[\munk]   dt + \int_{\R^d} \; f((n+\frac{1}{2})\dt)  \; d\mundk - \int_{\R^d} \; f(n\dt)  \; d\munk.
\end{split}
\end{equation}
We go back to the first term of \eqref{eq:sumint}. Notice that from \eqref{eq:sumint1}, we have
\begin{equation*}
\begin{split}
& \int_{n\dt}^{(n+\frac{1}{2})\dt} \; \left( \int_{\R^d} \; \partial_t f  \; d(\munk + 2(t-n\dt) h[\munk])  \right) dt + \int_{n\dt}^{(n+1)\dt} \; \int_{\R^d} \; f \; dh[\munk] dt \\
= & \int_{(n+\frac{1}{2})\dt}^{(n+1)\dt} \; \int_{\R^d} \; f \; dh[\munk] dt - \int_{n\dt}^{(n+\frac{1}{2})\dt} \;  \int_{\R^d} \; f  \; dh[\munk]   dt\\
& + \int_{\R^d} \; f((n+\frac{1}{2})\dt)  \; d\mundk - \int_{\R^d} \; f(n\dt)  \; d\munk\\
= & \int_{n\dt}^{(n+\frac{1}{2})\dt}  \int_{\R^d} \;(f(t+\frac{\dt}{2}) -f(t)) \; dh[\munk] dt 
 + \int_{\R^d} \; f((n+\frac{1}{2})\dt)  \; d\mundk - \int_{\R^d} \; f(n\dt)  \; d\munk \\
 = & \int_{n\dt}^{(n+\frac{1}{2})\dt}  \int_{\R^d} (\frac{\dt}{2} \partial_t f(t) + O(\dt^2)) \; dh[\munk] dt 
 + \int_{\R^d} \; f((n+\frac{1}{2})\dt)  \; d\mundk - \int_{\R^d} \; f(n\dt)  \; d\munk. 
\end{split}
\end{equation*}

Similarly, 
since $\Phi^{V[\munk]}_{2(t-(n+\frac{1}{2}))\dt}\#\mundk$ is solution at time $\tau=2(t-(n+\frac{1}{2}))\dt$ to 
$$
\partial_\tau\nu_\tau + \nabla \cdot (V[\munk]\nu_\tau) = 0, \quad \nu_0 = \mundk,
$$
it satisfies 
\[
\begin{split}
&\int_0^{\dt} \int_{\R^d} \partial_\tau f(\frac{\tau}{2}+(n+\frac{1}{2})\dt) d(\nu_\tau) d\tau + \int_0^{\dt} \int_{\R^d} \nabla f(\frac{\tau}{2}+(n+\frac{1}{2})\dt)\cdot V[\munk]  d(\nu_\tau) d\tau \\
= & \int_{\R^d} \; f((n+1)\dt) d\nu_{\dt} -  \int_{\R^d} \; f((n+\frac{1}{2})\dt) d\nu_0
\end{split}
\]
After the change of variables $t = \frac{\tau}{2}+(n+\frac{1}{2})\dt$, we obtain
\begin{equation*}
\begin{split}
& \int_{(n+\frac{1}{2})\dt}^{(n+1)\dt} \int_{\R^d} (\partial_t f(t) +2 \nabla f(t) \cdot V[\munk] ) d(\Phi^{V[\munk]}_{2(t-(n+\frac{1}{2}))\dt}\#\mundk) dt \\
= & \int_{\R^d} \; f((n+1)\dt) d\munnk -  \int_{\R^d} \; f((n+\frac{1}{2})\dt) d\mundk .
\end{split}
\end{equation*}

We now use this to evaluate the third term of \eqref{eq:sumint}. We have:
\begin{equation}\label{eq:sumint2}
\begin{split}
& \int_{(n+\frac{1}{2})\dt}^{(n+1)\dt} \int_{\R^d} (\partial_t f + \nabla f \cdot V[\munk] ) d(\Phi^{V[\munk]}_{2(t-(n+\frac{1}{2}))\dt}\#\mundk) dt \\
= & - \int_{(n+\frac{1}{2})\dt}^{(n+1)\dt} \int_{\R^d} \nabla f \cdot V[\munk]  d(\Phi^{V[\munk]}_{2(t-(n+\frac{1}{2}))\dt}\#\mundk) dt \\
& + \int_{\R^d} \; f((n+1)\dt) d\munnk -  \int_{\R^d} \; f((n+\frac{1}{2})\dt) d\mundk .
\end{split}
\end{equation}

Now adding together the second and third terms of \eqref{eq:sumint} and using \eqref{eq:sumint2}, we obtain: 
\begin{equation*}\label{eq:sumint3}
\begin{split}
 & \int_{n\dt}^{(n+\frac{1}{2})\dt} \; \left( \int_{\R^d} \; ( \nabla f \cdot V[\munk]) \; d(\munk + 2(t-n\dt) h[\munk])  \right) dt \\
& + \int_{(n+\frac{1}{2}\dt}^{(n+1)\dt} \; \left( \int_{\R^d} \; (\partial_t f + \nabla f \cdot V[\munk]) \; d(\Phi^{V[\munk]}_{2(t-(n+\frac{1}{2}))\dt}\#\mundk) \right) dt \\
= &  \int_{n\dt}^{(n+\frac{1}{2})\dt} \;  \int_{\R^d} \; \nabla f \cdot V[\munk] \; d\mu_t^k \; dt \\
& - \int_{(n+\frac{1}{2})\dt}^{(n+1)\dt} \int_{\R^d} \nabla f \cdot V[\munk]  d\mu_t^k \; dt + \int_{\R^d} \; f((n+1)\dt) d\munnk -  \int_{\R^d} \; f((n+\frac{1}{2})\dt) d\mundk. \\
\end{split}
\end{equation*}
Now, 
\begin{equation*}\label{eq:sumint4}
\begin{split}
 &  \int_{n\dt}^{(n+\frac{1}{2})\dt} \;  \int_{\R^d} \; \nabla f \cdot V[\munk] \; d\mu_t^k \; dt 
 - \int_{(n+\frac{1}{2})\dt}^{(n+1)\dt} \int_{\R^d} \nabla f \cdot V[\munk]  d\mu_t^k \; dt \\
 = & \int_{n\dt}^{(n+\frac{1}{2})\dt}  \int_{\R^d} \; \nabla f(t) \cdot V[\munk] \; d\mu_t^k \; dt 
 - \int_{n\dt}^{(n+\frac{1}{2})\dt} \int_{\R^d} \nabla f(t+\frac{\dt}{2}) \cdot V[\munk]  d\mu_{t+\frac{\dt}{2}}^k \; dt \\
 = & \int_{n\dt}^{(n+\frac{1}{2})\dt}  \int_{\R^d} \nabla f(t) \cdot V[\munk] \; d(\mu_t^k-\mu_{t+\frac{\dt}{2}}^k)  dt +
 \int_{n\dt}^{(n+\frac{1}{2})\dt} \int_{\R^d} (\nabla f(t) - \nabla f(t+\frac{\dt}{2}) ) \cdot V[\munk]  d\mu_{t+\frac{\dt}{2}}^k  dt \\
 = & \int_{n\dt}^{(n+\frac{1}{2})\dt}  \int_{\R^d} \nabla f(t) \cdot V[\munk] \; d(\mu_t^k-\mu_{(n+\frac{1}{2})\dt}^k)  dt +
 \int_{n\dt}^{(n+\frac{1}{2})\dt}  \int_{\R^d} \nabla f(t) \cdot V[\munk] \; d(\mu_{(n+\frac{1}{2})\dt}^k-\mu_{t+\frac{\dt}{2}}^k)  dt \\
& + \int_{n\dt}^{(n+\frac{1}{2})\dt} \int_{\R^d} (\nabla f(t) - \nabla f(t+\frac{\dt}{2}) ) \cdot V[\munk]  d\mu_{t+\frac{\dt}{2}}^k  dt.
\end{split}
\end{equation*}
The first term gives:
\begin{equation*}
\begin{split}
& \left| \int_{n\dt}^{(n+\frac{1}{2})\dt}  \int_{\R^d} \nabla f(t) \cdot V[\munk] \; d(\mu_t^k-\mu_{(n+\frac{1}{2})\dt}^k)  dt  \right|\\
= & \left| \int_{n\dt}^{(n+\frac{1}{2})\dt}  \int_{\R^d} \nabla f(t) \cdot V[\munk] \; 2((n+\frac{1}{2})\dt-t) dh[\munk]  dt  \right|\\
\leq &  M_V \|\nabla f \|_{L^\infty} 2 \left(\frac{\dt}{2}\right)^2 |h[\munk]| = M_V \bS \|\nabla f \|_{L^\infty} \dt^2.
 \end{split}
\end{equation*}

The second term gives:
\begin{equation*}
\begin{split}
& \left| \int_{n\dt}^{(n+\frac{1}{2})\dt}  \int_{\R^d} \nabla f(t) \cdot V[\munk] \; d(\mu_{(n+\frac{1}{2})\dt}^k-\mu_{t+\frac{\dt}{2}}^k)  dt \right|\\
\leq & \left| \int_{n\dt}^{(n+\frac{1}{2})\dt} L_1  \rho(\mu_{(n+\frac{1}{2})\dt}^k, \mu_{t+\frac{\dt}{2}}^k)  dt \right|
\leq L_1  \int_{n\dt}^{(n+\frac{1}{2})\dt} M_V (t+\frac{\dt}{2} - (n+\frac{1}{2})\dt )  dt \leq L_1 M_V \dt^2
 \end{split}
\end{equation*}
where, denoting by $L_1(t)$ the Lipschitz constant of the function $x\mapsto \nabla f(t,x) \cdot V[\munk](x)$, we define $L_1:=\sup_{t\in (0,T)} L_1(t)$. Notice that it is independent of $n$ and $k$ as seen in Proposition \ref{Prop:V}.
 
 Lastly, the third term gives:
 \begin{equation*}
\begin{split}
& \left| \int_{n\dt}^{(n+\frac{1}{2})\dt} \int_{\R^d} (\nabla f(t) - \nabla f(t+\frac{\dt}{2}) ) \cdot V[\munk]  d\mu_{t+\frac{\dt}{2}}^k  dt \right| \\
 \leq & \int_{n\dt}^{(n+\frac{1}{2})\dt} \int_{\R^d} \frac{\dt}{2} |\partial_t (\nabla f(t))|\, |V[\munk]|  d\mu_{t+\frac{\dt}{2}}^k  dt  \leq M_V \|\partial_t \nabla f\|_{L^\infty} \frac{\dt^2}{4}.
 \end{split}
\end{equation*}

We can finally go back to \eqref{eq:sumint}. 
\begin{equation*}
\begin{split}
& \int_{n\dt}^{(n+1)\dt} \; \left( \int_{\R^d} \; (\partial_t f + V[\munk]\cdot \nabla f) \; d\mu_t^k + \int_{\R^d} \; f \; dh[\munk] \right) dt\\
\leq & \int_{n\dt}^{(n+\frac{1}{2})\dt}  \int_{\R^d} (\frac{\dt}{2} \partial_t f(t) + O(\dt^2)) \; dh[\munk] dt 
 + \int_{\R^d} \; f((n+\frac{1}{2})\dt)  \; d\mundk - \int_{\R^d} \; f(n\dt)  \; d\munk \\
 &  + \int_{\R^d} \; f((n+1)\dt) d\munnk -  \int_{\R^d} \; f((n+\frac{1}{2})\dt) d\mundk + M_V(\bS \|\nabla f \|_{L^\infty} + L_1 + \frac{1}{4}\|\partial_t \nabla f\|_{L^\infty}) \dt^2 \\
 \leq & \int_{\R^d} \; f((n+1)\dt) d\munnk - \int_{\R^d} \; f(n\dt)  \; dh[\munk] + C \dt^2, 
\end{split}
\end{equation*}
with $C:= 2\bS \|\partial_t f\|_{L^\infty}+ M_V(\bS \|\nabla f \|_{L^\infty} + L_1 + \frac{1}{4}\|\partial_t \nabla f\|_{L^\infty})$.
Thus, 
\begin{equation*}
\begin{split}
 & \lim_{k\rightarrow\infty} \left| \sum_{n=0}^{2^k-1}\int_{n\dt}^{(n+1)\dt} \; \left( \int_{\R^d} \; (\partial_t f + V[\munk]\cdot \nabla f) \; d\mu_t^k + \int_{\R^d} \; f \; dh[\munk] \right) dt \right | \\
 \leq & \lim_{k\rightarrow\infty} C \sum_{n=0}^{2^k-1}\dt^2 = \lim_{k\rightarrow\infty} C T 2^{-k}  = 0.
\end{split}
\end{equation*}
\end{proof}

We can now prove the following: 

\begin{prop}\label{Prop:existence}
The limit measure $\bmu_t=\lim_{k\rightarrow\infty} $ is a weak solution to \eqref{eq:transportsource}. Moreover, $\bmu_t\in \PcR$ and for all $R>0$, there exists $R_T>0$ such that if $\supp(\bmu_0)\subset B(0,R)$, for all $t\in [0,T]$, $\supp(\bmu_t)\subset B(0,R_T)$.
\end{prop}

\begin{proof}
We will prove that for all  $f\in C^\infty_c((0,T)\times\R^d)$,
\begin{equation}\label{eq:equalityweak}
\begin{split}
\lim_{k\rightarrow\infty} \sum_{n=0}^{2^k-1}\int_{n\dt}^{(n+1)\dt} \; \left( \int_{\R^d} \; (\partial_t f + V[\munk]\cdot \nabla f) \; d\mu_t^k + \int_{\R^d} \; f \; dh[\munk] \right)dt \\
- \int_0^T \left(\int_{\R^d} (\partial_t f +  V[\bmu_t]\cdot \nabla f) \; d\bmu_t + \int_{\R^d} \; f \; dh[\bmu_t] \right)dt = 0. 
\end{split}
\end{equation}
First, denoting by $F_1:= \sup_{[0,T]} \|\partial_t f(t,\cdot)\|_{\Lip}+ \|\partial_t f\|_{L^\infty((0,T)\times\R^d)}$, observe that 
\begin{equation*}
\begin{split}
& \sum_{n=0}^{2^k-1}\int_{n\dt}^{(n+1)\dt} \; \int_{\R^d} \; \partial_t f \; d(\mu_t^k - \bmu_t )dt = 
 F_1 \sum_{n=0}^{2^k-1}\int_{n\dt}^{(n+1)\dt} \; \int_{\R^d} \; \frac{\partial_t f }{F_1} d(\mu_t^k - \bmu_t )dt  \\
 \leq & F_1 \sum_{n=0}^{2^k-1}\int_{n\dt}^{(n+1)\dt} \; \left(\sup_{f\in \C_c^{0,\Lip}, \|f\|_\Lip\leq 1, \|f\|_{\infty}\leq 1} \int_{\R^d} f\; d(\mu_t^k - \bmu_t ) \right) dt
 = F_1 \sum_{n=0}^{2^k-1}\int_{n\dt}^{(n+1)\dt} \rho(\mu_t^k, \bmu_t ) dt \\
 \leq & F_1 T \; \D(\mu^k, \bmu) \xrightarrow[k\rightarrow\infty]{} 0.
\end{split}
\end{equation*}
Secondly, denoting by $F_2:= \sup_{[0,T]} \|f(t,\cdot)\|_{\Lip}+ \|f\|_{L^\infty((0,T)\times\R^d)} $,
\begin{equation*}
\begin{split}
& \int_{\R^d} \;  f \; d(h[\munk]-h[\bmu_t]) = F_2 \int_{\R^d} \;  \frac{f}{F_2} \; d(h[\munk]-h[\bmu_t]) \\
 \leq & F_2 \sup_{f\in \C_c^{0,\Lip}, \|f\|_\Lip\leq 1, \|f\|_{L^\infty(\R^d)}\leq 1} \int_{\R^d} f\; d(h[\munk] - h[\bmu_t] ) = F_2 \rho(h[\munk] , h[\bmu_t] ) \leq F_2 L_h \rho(\munk , \bmu_t ) \\
 \leq & F_2 L_h (\rho(\munk , \mu_t^k )+\rho(\mu_t^k , \bmu_t )) \leq F_2 L_h ((\bS+M_V) \dt +\D(\mu_t^k , \bmu_t )) 
\end{split}
\end{equation*}
from Equation \eqref{eq:eststep}. Hence, 
\begin{equation*}
\begin{split}
& \sum_{n=0}^{2^k-1}\int_{n\dt}^{(n+1)\dt} \; \int_{\R^d} \;  f \; d(h[\munk]-h[\bmu_t]) dt \leq F_2 L_h \sum_{n=0}^{2^k-1}\int_{n\dt}^{(n+1)\dt} ((\bS+M_V) \dt +\D(\mu_t^k , \bmu_t )) dt \\
\leq & (\bS+M_V) \sum_{n=0}^{2^k-1} \dt^2 + T \D(\mu_t^k , \bmu_t )) = 2^{-k} T (\bS+M_V) + T \D(\mu_t^k , \bmu_t )) \xrightarrow[k\rightarrow\infty]{} 0.
\end{split}
\end{equation*}
Thirdly, denoting by $F_3:=\sup_{[0,T]} \|\nabla f(t,\cdot)\|_{\Lip}+ \|\nabla f\|_{L^\infty((0,T)\times\R^d)}$,
\begin{equation*}
\begin{split}
& \int_{\R^d} \;  V[\munk]\cdot \nabla f \; d\mu_t^k - \int_{\R^d}   V[\bmu_t]\cdot \nabla f \; d\bmu_t \\
 = &   \int_{\R^d} \;  V[\munk]\cdot \nabla f \; d(\mu_t^k-\bmu_t) + \int_{\R^d}  (V[\munk]- V[\mutk])\cdot \nabla f \; d\bmu_t
 + \int_{\R^d}  (V[\mutk]-V[\bmu_t])\cdot \nabla f \; d\bmu_t \\
 \leq &  F_3 (M_V+ L_V) \rho(\mu_t^k,\bmu_t) + F_3 L_V (\rho(\munk,\mutk)+\rho(\mutk,\bmu_t))\\
 \leq & F_3 (M_V+ 2L_V)  \rho(\mu_t^k,\bmu_t)  + F_3 L_V (\bS+M_V) \dt.
\end{split}
\end{equation*}
Hence, 
\begin{equation*}
\begin{split}
&  \sum_{n=0}^{2^k-1}\int_{n\dt}^{(n+1)\dt} \;\int_{\R^d} \;  V[\munk]\cdot \nabla f \; d\mu_t^k - \int_{\R^d}   V[\bmu_t]\cdot \nabla f \; d\bmu_t \\
 \leq & F_3 (M_V+ 2L_V) T \rho(\mu_t^k,\bmu_t)  + F_3 L_V (\bS+M_V) 2^{-k} T \xrightarrow[k\rightarrow\infty]{} 0.
\end{split}
\end{equation*}
We conclude that \eqref{eq:equalityweak} holds, and from Lemma \ref{lemma:convergenceweak}, we obtain:
$$
\int_0^T \left(\int_{\R^d} (\partial_t f +  V[\bmu_t]\cdot \nabla f) \; d\bmu_t + \int_{\R^d} \; f \; dh[\bmu_t] \right)dt =0.
$$
As remarked in \cite{PiccoliRossiTournus20}, this weak formulation is equivalent to the Definition \ref{Def:weaksol}.
This proves that $\bmu_t$ is a weak solution to \eqref{eq:transportsource}. 
The compactness of its support can be deduced from Prop. \ref{Prop:suppmu}.
\end{proof}

\begin{prop}\label{Prop:conti}
Let $\mu, \nu \in C([0,T], \PcR)$ be two solutions to \eqref{eq:transportsource} with initial conditions $\mu_0, \nu_0$.
There exists a constant $C>0$ such that for all $t\in [0,T]$,
\begin{equation*}
\rho(\mu_t,\nu_t) \leq e^{Ct} \; \rho(\mu_0,\nu_0).
\end{equation*}
In particular, this implies uniqueness of the solution to \eqref{eq:transportsource}.
\end{prop}

\begin{proof}
Let $\mu, \nu \in C([0,T], \PcR)$ be two solutions to \eqref{eq:transportsource} with initial conditions $\mu_0, \nu_0$.
Let $\varepsilon(t) = \rho(\mu_t,\nu_t)$.
Then
\begin{equation}\label{eq:epst-tau}
\begin{split}
\varepsilon(t+\tau) = \rho(\mu_{t+\tau},\nu_{t+\tau}) 
\leq & \;
\rho(\mu_{t+\tau}, \Phi_\tau^{V[\mu_t]}\#(\mu_t+\tau h[\mu_t])) 
+\rho(\nu_{t+\tau}, \Phi_\tau^{V[\nu_t]}\#(\nu_t+\tau h[\nu_t])) \\
&+\rho( \Phi_\tau^{V[\mu_t]}\#(\mu_t+\tau h[\mu_t]),  \Phi_\tau^{V[\nu_t]}\#(\nu_t+\tau h[\nu_t]) ).
\end{split}
\end{equation}
From Prop \ref{prop:flowproperties}, it holds:
\begin{equation}\label{eq:epst-tau3}
\begin{split}
&\rho( \Phi_\tau^{V[\mu_t]}\#(\mu_t+\tau h[\mu_t]),  \Phi_\tau^{V[\nu_t]}\#(\nu_t+\tau h[\nu_t]) ) \\
\leq \; & (1+2 L\tau) \; \rho(\mu_t+\tau h[\mu_t], \nu_t+\tau h[\nu_t]) +
\min\{|\mu_t+\tau h[\mu_t]|, |\nu_t+\tau h[\nu_t]|\} 2\tau L_V \rho(\mu_t,\nu_t) \\
\leq \; & (1+2 L\tau) (1+\tau L_h) \;  \rho(\mu_t , \nu_t ) +
(1+\tau \bS)\; 2\tau L_V \rho(\mu_t,\nu_t) \\
\leq \; &(1+ (2\La+L_h+2 L_V) \tau + 2(\La L_h+L_V \bS)\tau^2 ) \; \rho(\mu_t,\nu_t) \leq (1+ 2(2\La+L_h+2 L_V) \tau )\;  \rho(\mu_t,\nu_t).
\end{split}
\end{equation}
For the first and the second term, we prove that any solution $\mu$ to \eqref{eq:transportsource} satisfies the operator-splitting estimate:
\begin{equation}\label{eq:Ktausquare}
\forall (t,\tau)\in [0,T]\times[0,T-t],  \quad \rho(\mu_{t+\tau}, \Phi_\tau^{V[\mu_t]}\#\mu_t + \tau h[\mu_t]) \leq K\; \tau^2.
\end{equation}
We begin by proving \eqref{eq:Ktausquare} for solutions to the numerical scheme $\S$. Let $k\in\N$ and $\mutk$ be the solution to $\S$ with time-step $\dt = 2^{-k}T$ and initial condition $\mu_0$.
For simplicity, we assume that $t=n\dt$ and $\tau = l\dt$, with $(n,l)\in\N^2$, and we study the distance
\begin{equation*}
\begin{split}
D_l := \rho(\mu^k_{(n+l)\dt}, \Phi_{l\dt}^{V[\munk]}\#(\munk+ l\dt\; h[\munk])).
\end{split}
\end{equation*}
Notice that by definition of the numerical scheme, for $l=1$, $D_1=0$.
For $l=2$, denoting $H_m^j = h[\mu_{m\dt}^j]$ and $P_m^j = \Phi^{V[\mu_{m\dt}^j]}_{\dt}$, and using the properties listed in Propositions \ref{prop:flowproperties}, \ref{Prop:V} and \ref{Prop:h}, it holds:
\begin{equation*}
\begin{split}
D_2 = \; & \rho(\mu^k_{(n+2)\dt}, \Phi_{2\dt}^{V[\munk]}\#(\munk+ 2\dt\; h[\munk])) \\
 = \; & \rho(P_{n+1}^k\#(P_n^k\#(\munk+ \dt\; H_n^k)+ \dt\; H_{n+1}^k), P^k_n\#P^k_n\#(\munk+ 2\dt\; H_n^k)) \\
\leq \; & (1+2\La \dt) \; \rho(P_n^k\#(\munk+ \dt\; H_n^k)+ \dt\; H_{n+1}^k, P^k_n\#(\munk+ 2\dt\; H_n^k)) \\
& +2\dt L_V \rho(\munnk, \munk)\\
\leq \; & (1+2\La \dt)  \rho (\dt H_{n+1}^k, \dt P^k_n\# H_n^k) + 2\dt L_V \rho(\munnk, \munk)  \\
\leq \; & (1+2\La \dt) \dt \;( \rho ( H_{n}^k, P^k_n\# H_n^k) + \rho ( H_{n+1}^k, H_n^k) ) + 2\dt L_V \rho(\munnk, \munk) \\
\leq \; & (1+2\La \dt) ( \dt^2 M_V |H_n^k| + L_h\dt \rho(\munnk, \munk)) + 2\dt L_V \rho(\munnk, \munk) \\
\leq \; & (1+2\La \dt) ( \dt^2 M_V \bS + L_h \dt \rho(\munnk, \munk)) + 2\dt L_V \rho(\munnk, \munk) \\
\leq \; & (2(L_h+2 L_V) \dt \; \rho(\munnk, \munk) + 2 M_V \bS \dt^2  \leq 2 ( (L_h+2 L_V)(\bS+M_V)+ M_V \bS) \dt^2 \leq K\dt^2
\end{split}
\end{equation*}
where the last equality was obtained using \eqref{eq:eststepts}
and defining $K := 2 ( (L_h+2 L_V)(\bS+M_V)+M_V \bS)$. Let us now suppose that for some $l\in\N$,  $D_l\leq K(l-1)^2\dt^2$.
We compute
\begin{equation*}
\begin{split}
D_{l+1} = \; & \rho(\mu^k_{(n+l+1)\dt}, \Phi_{(l+1)\dt}^{V[\munk]}\#(\munk+ (l+1)\dt\; h[\munk])) \\
= \; & \rho(P_{n+l}^k\# (\mu_{(n+l)\dt}^k + \dt \;H_{n+l}^k), P_n^k\#\Phi_{l\dt}^{V[\munk]}\#(\munk+ l\dt\; H_{n}^k+\dt\; H_{n}^k)) \\
\leq \; & \rho(P_{n+l}^k\# \mu_{(n+l)\dt}^k , P_n^k\#\Phi_{l\dt}^{V[\munk]}\#(\munk+ l\dt\; H_{n}^k))
+ \dt \rho( P_{n+l}^k\# \;H_{n+l}^k, P_n^k\#\Phi_{l\dt}^{V[\munk]}\# H_{n}^k) \\
\leq \; & (1+2 \La \dt) \rho(\mu_{(n+l)\dt}^k , \Phi_{l\dt}^{V[\munk]}\#(\munk+ l\dt\; H_{n}^k)) \\
& + \min\{|\mu_{(n+l)\dt}^k|, |\Phi_{l\dt}^{V[\munk]}\#(\munk+ l\dt\; H_{n}^k)|\} 2\dt L_V \rho(\mu_{(n+l)\dt}^k, \munk)\\
& + \dt \; (1+2 \La \dt) \rho( H_{n+l}^k, \Phi_{l\dt}^{V[\munk]}\# H_{n}^k) \\
& + \dt \min\{|h[\mu_{(n+l)\dt}^k|, |\Phi_{l\dt}^{V[\munk]}\#h[\munk]|\} 2\dt L_V \rho(\mu_{(n+l)\dt}^k, \munk).
\end{split}
\end{equation*}
From Prop. \ref{Prop:muproba}, we know that for $k$ large enough, $\Phi_{l\dt}^{V[\munk]}\#(\munk+ l\dt\; H_{n}^k)\in\PR$, thus
$ |\Phi_{l\dt}^{V[\munk]}\#(\munk+ l\dt\; H_{n}^k)|=|\mu_{(n+l)\dt}^k|=1$.
Now using the fact that $\rho(\mu_{(n+l)\dt}^k, \munk)\leq l\dt(M_V+\bS)$, we compute:
\begin{equation*}
\begin{split}
D_{l+1}\leq \; & (1+2 \La \dt) D_l + 2\dt L_V \rho(\mu_{(n+l)\dt}^k, \munk) \\
& + \dt (1+2\La \dt) ( \rho(h[\mu_{(n+l)\dt}^k], h[\munk]) + l\dt M_V |h[\munk]| ) + 2 L_V \dt^2 \;\rho(\mu_{(n+l)\dt}^k, \munk) \\
\leq \; & (1+2 \La\dt) K ((l-1)^2\dt^2) + 2\dt L_V l \dt (M_V+\bS) \\
& +  \dt (1+2\La \dt) (L_h l\dt(M_V+\bS) + l \dt M_V \bS) + 2 L_V \dt^2 l \dt(M_V+\bS) \\
\leq \; & \dt^2 [ K (l-1)^2 + l( (2 L_V + L_h)(M_V+\bS) + M_V\bS)] + O(\dt^3) 
\leq  \dt^2 [ K (l-1)^2 + K l] \leq K l^2\dt^2.
\end{split}
\end{equation*}
Thus, by induction, $\rho(\mu^k_{(n+l)\dt}, \Phi_{l\dt}^{V[\munk]}\#(\munk+ l\dt\; h[\munk])) \leq K (l\dt)^2$
and similarly we can prove that 
$$
\forall (t,\tau)\in [0,T]\times[0,T-t],  \quad \rho(\mu^k_{t+\tau}, \Phi_\tau^{V[\mu_t]}\#\mu^k_t + \tau h[\mu^k_t]) \leq K\; \tau^2.
$$
Hence, 
\begin{equation*}
\begin{split}
\rho(\mu_{t+\tau}, \Phi_\tau^{V[\mu_t]}\#\mu_t + \tau h[\mu_t]) \leq &  \rho(\mu^k_{t+\tau}, \Phi_\tau^{V[\mu_t]}\#\mu^k_t + \tau h[\mu^k_t]) + \rho(\mu_{t+\tau},\mu^k_{t+\tau})\\
& + \rho(\Phi_\tau^{V[\mu_t]}\#\mu_t + \tau h[\mu_t], \Phi_\tau^{V[\mu_t]}\#\mu^k_t + \tau h[\mu^k_t])
\end{split}
\end{equation*}
and by taking the limit $k\rightarrow\infty$,
$\rho(\mu_{t+\tau}, \Phi_\tau^{V[\mu_t]}\#\mu_t + \tau h[\mu_t])  \leq K \tau^2$, which proves \eqref{eq:Ktausquare}.

Coming back to \eqref{eq:epst-tau}, and using \eqref{eq:epst-tau3} and \eqref{eq:Ktausquare}, it holds
\begin{equation*}
\begin{split}
\varepsilon(t+\tau) \leq (1+ 2(2\La+2L_V+L_h) \tau )\;  \varepsilon(t) + 2 K\tau^2.
\end{split}
\end{equation*}
Then
\begin{equation*}
\begin{split}
\frac{\varepsilon(t+\tau)-\varepsilon(t) }{\tau} \leq 2(2\La+2L_V+L_h)\varepsilon(t) + 2 K\tau.
\end{split}
\end{equation*}
which proves that $\varepsilon$ is differentiable and that 
\begin{equation*}
\begin{split}
\varepsilon'(t) \leq 2(2\La+2L_V+L_h) \varepsilon(t).
\end{split}
\end{equation*}
This implies that 
\begin{equation*}
\varepsilon(t) \leq  \varepsilon(0) e^{2(2\La+2L_V+L_h)t}.
\end{equation*}
This proves continuity with respect to the initial data, i.e. uniqueness of the solution.
\end{proof}

We have thus proven Theorem \ref{Th:existenceuniqueness}: Existence was obtained as the limit of the numerical scheme $\S$ in Proposition \ref{Prop:existence}; Uniqueness comes from Proposition \ref{Prop:conti}.

We saw in Section \ref{Sec:Compa} that the Bounded Lipschitz distance and the $1$-Wasserstein distance are equivalent on the set of probability measures with uniformly compact support. This allows us to state the following:

\begin{corollary}\label{Col:W1}
Let $\mu, \nu \in C([0,T], \PcR)$ be two solutions to \eqref{eq:transportsource} with initial conditions $\mu_0, \nu_0$ satisfying $\supp(\mu_0)\cup\supp(\nu_0)\subset B(0,R)$.
There exist constants $C>0$ and $C_{R_T}>0$ such that for all $t\in [0,T]$,
\begin{equation*}
W_1(\mu_t,\nu_t) \leq C_{R_T} e^{Ct} \; W_1(\mu_0,\nu_0).
\end{equation*}
Furthermore, for all $p\in\N^*$,
\begin{equation*}
W_p(\mu_t,\nu_t) \leq (2R)^{\frac{p-1}{p}} C_{R_T}^{\frac{1}{p}} e^{\frac{C}{p}t} \; W_p(\mu_0,\nu_0)^{\frac{1}{p}}.
\end{equation*}
\end{corollary}

\begin{proof}
Let $R>0$ such that $\supp(\mu_0)\cup\supp(\nu_0)\in B(0,R)$. From Prop. \ref{Prop:existence}, there exists $R_T>0$ such that for all $t\in [0,T]$, $\supp(\mu_t)\cup\supp(\nu_t)\subset B(0,R_T)$.
Putting together Prop. \ref{Prop:conti}, equation \eqref{eq:BLleqWass} and Prop. \ref{Prop:WassleqBL},
\[
W_1(\mu_t,\nu_t) \leq C_{R_T} \rho(\mu_t,\nu_t) \leq C_{R_T} e^{Ct} \; \rho(\mu_0,\nu_0) \leq C_{R_T} e^{Ct} W_1(\mu_0,\nu_0),
\]
where $C_{R_T}=\max(1,R_T)$.
Moreover,
for all $p\in\N^*$, from equation \eqref{eq:WasspleqWassq} and Prop. \ref{Prop:WassqleqWass1}, it holds
\[
W_p(\mu_t^N,\mu_t) \leq (2R)^{\frac{p-1}{p}} W_1(\mu_t^N,\mu_t)^{\frac{1}{p}} \leq (2R)^{\frac{p-1}{p}} C_{R_T}^{\frac{1}{p}} e^{\frac{C}{p}t} W_1(\mu_0^N,\mu_0)^{\frac{1}{p}}\leq  (2R)^{\frac{p-1}{p}} C_{R_T}^{\frac{1}{p}} e^{\frac{C}{p}t} W_p(\mu_0^N,\mu_0)^{\frac{1}{p}} .
\]
\end{proof}

\section{Convergence to the macroscopic model}

Having proven the well-posedness of both the microscopic and macroscopic models, we are now in a position to prove the convergence result stated in Theorem \ref{th:convergence} that is central to this paper.
The proof, as for the now classical proof of convergence of the microscopic dynamics without weights \eqref{eq:syst-micro-gen} to the non-local transport PDE \eqref{eq:transport-gen} (see \cite{Dobrushin79}), relies on two ingredients: 
the fact that the empirical measure satisfies the PDE and the continuity of the solution with respect to the initial data.
We begin by defining the empirical measure for our microscopic system with weight dynamics and prove that it does satisfy the PDE \eqref{eq:transportsource}.

\subsection{From microscopic to macroscopic via the empirical measure}

The fact that \eqref{eq:syst-micro} preserves indistinguishability allows us to define a generalized version of the empirical measure. For all $N\in\N$ and $(x,m)\in\mathcal{C}([0,T];(\R^d)^N\times\R^{N})$ solution to \eqref{eq:syst-micro}, let
\begin{equation}\label{eq:empmeas}
\mu^N_t = \frac{1}{M}\sum_{i=1}^N m_i(t) \delta_{x_i(t)}
\end{equation}
be the \emph{generalized empirical measure}.
From Prop. \ref{Prop:m}, we know that for all $t\in [0,T]$, $\mu_t\in \PR$.
We can prove the following:

\begin{prop}
Let $(x,m)\in\mathcal{C}([0,T];(\R^d)^N\times\R^{N})$ be a solution to \eqref{eq:syst-micro}, and let $\mu^N\in C([0,T];\PR)$ denote the corresponding empirical measure, given by \eqref{eq:empmeas}. 
Then, $\mu^N$ is a weak solution to \eqref{eq:transportsource}.
\end{prop}
\begin{proof}
We show that $\mu^N_t$ satisfies \eqref{eq:transportsourceweak}. 
Let $f\in C_c^\infty(\R^d)$.
Substituting $\mu$ by $\mu^N$ in the left-hand side of~\eqref{eq:transportsourceweak}, we obtain 
\begin{equation}\label{eq:empeq1}
\begin{split}
\frac{d}{dt}\int_{\R^d} f(x) d\mu^N_t(x) & = \frac{d}{dt} \left[ \frac{1}{M} \sum_{i=1}^N m_i(t) f(x_i(t))\right] \\
& = 
\frac{1}{M} \sum_{i=1}^N \dot m_i(t) f(x_i(t)) + \frac{1}{M} \sum_{i=1}^N m_i(t) \nabla f(x_i(t))\cdot \dot x_i(t).
\end{split}
\end{equation}
The first part of the right-hand side of \eqref{eq:transportsourceweak} gives
\begin{equation}\label{eq:empeq2}
\begin{split}
 \int_{\R^d} V[\mu_t]\cdot \nabla f(x) d\mu^N_t(x) = & \int_{\R^d} \int_{\R^d} \phi(y-x) \cdot \nabla f(x) \, d\mu^N_t(y) \, d\mu^N_t(x) \\
= & \frac{1}{M^2}  \sum_{i=1}^N \sum_{j=1}^N m_i m_j \phi(x_j-x_i)\cdot \nabla f(x_i)
=  \frac{1}{M} \sum_{i=1}^N m_i \nabla f(x_i)\cdot \dot x_i.
\end{split}
\end{equation}
where the last equality comes from the fact that $x$ is a solution to \eqref{eq:syst-micro}.
The second part of the right-hand side of \eqref{eq:transportsourceweak} gives:
\begin{equation}\label{eq:empeq3}
\begin{split}
\int_{\R^d}f(x) &dh[\mu^N_t](x) = \int_{\R^d}f(x) \left(\int_{(\R^d)^q} S(x,y_1,\cdots,y_q) d\mu^N_t(y_1)\cdots d\mu^N_t(y_q)\right) d\mu^N_t(x)\\
& = \frac{1}{M} \sum_{i=1}^N m_i f(x_i)  \frac{1}{M^q} \sum_{j_1=1}^N \cdots \sum_{j_q=1}^N m_{j_1} \cdots m_{j_q} S(x_i, x_{j_1}, \cdots x_{j_q})
= \frac{1}{M} \sum_{i=1}^N \dot m_i f(x_i).
\end{split}
\end{equation}
where the last equality comes from the fact that $m$ is a solution to \eqref{eq:syst-micro}.
Putting \eqref{eq:empeq1}, \eqref{eq:empeq2} and \eqref{eq:empeq3} together, we deduce that $\mu^N_t$ satisfies \eqref{eq:transportsourceweak}, thus it is a weak solution to \eqref{eq:transportsource}.
\end{proof}

\subsection{Convergence}

We are finally equipped to prove Theorem \ref{th:convergence}, that we state again here in its full form:

\begin{maintheorem}
Let $T>0$, $q\in \N$ and $M>0$.
For each $N\in\N$, let $(x_i^{N,0}, m^{N,0}_i)_{i\in\elts}\in (\R^d)^N\times (\R^{+*})^N$ such that $\sum_{i=1}^N m^{N,0}_i = M$.
Let $\phi\in C(\R^d;\R^d)$ satisfying Hyp. \ref{hyp:abar} and let $S\in C((\R^d)^{q+1};\R)$ satisfying Hyp. \ref{hyp:S}.
For all $t\in [0,T]$, let $t\mapsto(x^N_i(t), m^N_i(t))_{i\in\elts}$ be the solution to 
\begin{equation*}
\begin{cases}
\displaystyle \dot{x}_i =  \frac{1}{M} \sum\limits_{j=1}^N m_j \phi \left(x_j-x_i\right), \quad x_i(0)=x_i^{N,0} \\
\displaystyle \dot{m}_i = m_i\frac{1}{M^q} \sum_{j_1=1}^N \cdots \sum_{j_q=1}^N m_{j_1}\cdots m_{j_q} S(x_i, x_{j_1}, \cdots x_{j_q}), \quad m_i(0)=m_i^{N,0},
\end{cases}
\end{equation*}
 and let $\mu^N_t:= \frac{1}{M}\sum_{i=1}^N m_i^N(t) \delta_{x_i^N(t)}\in \PcR$ be the corresponding empirical measure.
Let $\mathscr{D}(\cdot,\cdot)$ denote either the Bounded Lipschitz distance $\rho(\cdot,\cdot)$ or any of the Wasserstein distances $W_p(\cdot,\cdot)$ for $p\in\N^*$.
\[ \text{If there exists } \mu_0 \in \PcR \text{ s.t. } 
\lim_{N\rightarrow\infty} \mathscr{D}(\mu^N_0,\mu_0) = 0, \quad
\text{ then for all }  t\in [0,T], \quad 
\lim_{N\rightarrow\infty} \mathscr{D}(\mu^N_t,\mu_t) = 0, 
\]
where $\mu_t$ is the solution to the transport equation with source 
\begin{equation*}
\partial_t \mt(x) + \nabla\cdot \left(\int_{\R^d} \ba(x-y) d\mu_t(y) \; \mt(x)\right) = \left(\int_{(\R^d)^q} S(x,y_1,\cdots,y_q) d\mu_t(y_1)\cdots d\mu_t(y_q)\right) \mu_t(x),
\end{equation*}
with initial data $\mu_{t=0}=\mu_0$.
\end{maintheorem}

\begin{proof}
Since $\mu_t^N$ and $\mu_t$ are both weak solutions to \eqref{eq:transportsource}, from Prop. \ref{Prop:conti}, there exists $C>0$ such that
\[
\rho(\mu_t^N,\mu_t)\leq e^{Ct} \rho(\mu_0^N,\mu_0)
\]
and the result follows immediately for $\mathscr{D}=\rho$.\\
Let $R<0$ such that $\supp(\mu_0)\cup \supp(\mu_0^N)\subset B(0,R)$ for all $N\in\N$. From Corollary \ref{Col:W1}, there exists $C_{R_T}>0$ depending on $T$ and $R$ such that 
for all $p\in\N^*$,
\[
W_p(\mu_t^N,\mu_t) \leq (2R)^{\frac{p-1}{p}} C_{R_T}^{\frac{1}{p}} e^{\frac{C}{p}t} W_p(\mu_0^N,\mu_0)^{\frac{1}{p}} 
\]
and the result follows for $\mathscr{D}=W_p$.
\end{proof}

\section{Numerical simulations}

To illustrate our convergence result, we provide numerical simulations for a specific model. We also refer the reader to the paper \cite{AyiPouradierDuteil20} for numerical simulations with a different model.

We recall the first model (M1) proposed in \cite{McQuadePiccoliPouradierDuteil19}, ``increasing weight by pairwise competition'':
\begin{equation*}
\begin{cases}
\displaystyle \dot{x}_i(t) = \frac{1}{M} \sum\limits_{j=1}^N m_j(t) \ba(x_j(t)-x_i(t)), \quad x_i(0)=x_i^0 \\
\displaystyle \dot{m}_i(t) = \frac{1}{M} m_i(t) \sum_{j=1}^N m_j(t) \beta \langle \frac{\dot x_i(t)+\dot x_j(t)}{2}, u_{ji}\rangle, \quad m_i(0)=m_i^0
\end{cases}
\end{equation*}
where $u_{ji}$ is the unit vector in the direction $x_i-x_j$ and $\beta$ is a constant.
 
With this choice of model, the evolution of each agent's weight depends on the dynamics of the midpoints $(x_i+x_j)/2$ between $x_i$ and each other agent at position $x_j$.
More specifically, if the midpoint $(x_i+x_j)/2$ moves in the direction of $x_i$, i.e. $\langle \frac{\dot x_i+\dot x_j}{2}, u_{ji}\rangle >0$, then the weight $m_i$ increases proportionally to $m_j$. If, on the other hand, $(x_i+x_j)/2$ moves away from $x_i$ and towards $x_j$, the weight $m_i$ decreases by the same proportion. 

In order to ensure continuity, we slightly modify the model and replace $u_{ji}$ by a function $h(x_i-x_j)$, where $h\in \Lip(\R^d;\R^d)$ is non-decreasing and satisfies the following properties: 
\begin{itemize}
\item $h(y) = \tilde{h}(|y|) y$ for some $\tilde h\in C(\R^+;\R^+)$
\item $\lim_{|y|\rightarrow\infty} h(y) = \frac{y}{|y|}$
\end{itemize}

Then, by replacing $\dot x_i$ and $\dot x_j$ by their expressions, the system can be written as: 
\begin{equation}\label{eq:microexample}
\begin{cases}
\displaystyle \dot{x}_i = \frac{1}{M} \sum\limits_{j=1}^N m_j \ba(x_j-x_i), \quad x_i(0)=x_i^0 \\
\displaystyle \dot{m}_i = \frac{1}{M^2} m_i \sum_{j=1}^N \sum_{k=1}^N m_j m_k \; \beta \; \langle \frac{\ba(x_k-x_i)+\ba(x_k-x_j)}{2}, h(x_i-x_j)\rangle, \quad m_i(0)=m_i^0.
\end{cases}
\end{equation}
Notice that it is in the form of System \eqref{eq:syst-micro}, with $q = 2$ and $S\in C((\R^d)^3;\R)$ defined by  
\[ S(x,y,z) = \beta \; \langle \frac{\ba(z-x)+\ba(z-y)}{2}, h(x-y)\rangle. \]
One easily sees that $S(x,y,z) = -S(y,x,z)$, thus $S$ satisfies \eqref{eq:condS}.
Furthermore, for every $R_T>0$, there exists $\bS$ such that for all $x,y,z\in B(0,R_T)$, $S(x,y,z)\leq \bS$, hence condition \eqref{eq:psiSkbound} is satisfied in a relaxed form.
Lastly, it is simple to check that as long as $\phi\in \Lip(\R^d; \R^d)$, $S\in \Lip((\R^d)^3;\R)$ thus $S$ satisfies \eqref{eq:psiSklip}.

We can then apply Theorem \ref{th:convergence}.

Consider $\mu_0\in\P(\R)$. For simplicity purposes, for the numerical simulations we take $\mu_0$ supported on $[0,1]$ and absolutely continuous with respect to the Lebesgue measure.
For a given $N\in\N$, we define: 
\[
x_i^{N,0} := \frac{i}{N}, \qquad m_i^{N,0} := \int_{\frac{i-1}{N}}^{\frac{i}{N}} d\mu_0, \qquad \qquad i\in\elts.
\]
We then have convergence of the empirical measures $\mu_0^N$ to $\mu_0$ when $N$ goes to infinity.
According to Theorem \ref{th:convergence}, for all $t\in [0,T]$, $\mu_t^N \rightharpoonup \mu_t $, where $\mu_t$ is the solution to the transport equation with source
\begin{equation}\label{eq:MFexample}
\partial_t \mu_t(x) + \partial_x \left( \int_{\R} \phi(y-x) d\mu_t(y) \; \mu_t(x)\right) = \left( \int_{\R^2} S(x,y,z) d\mu_t(y)d\mu_t(z) \right) \mu_t(x)
\end{equation}

Figures \ref{Fig:Pos} and \ref{Fig:Mean-field} illustrate this convergence for the specific choices : 
\begin{itemize}
\item $\phi:=\phi_{0.2}$, where for all $R>0$,
\[
\phi_R:\delta\mapsto 
\left\{
\begin{array}{ll} 
\frac{\delta}{|\delta|} \sin^2(\frac{\pi}{R}|\delta|) \quad & \text{ if } |\delta|\leq R  \\
0 \quad & \text{ if } |\delta|> R
\end{array}
\right.
\]
\item $h:\delta\mapsto \arctan(|\delta|) \frac{\delta}{|\delta|}$
\item $\beta := 100$
\item $M:=N$.
\item $d\mu_0(x):= \frac{f(x)}{F} dx$, with $f(x) := [\frac{3.5}{\sqrt{0.4\pi}}\exp(-\frac{5(x-0.25)^2}{4})+\frac{1}{\sqrt{0.4\pi}}\exp(-\frac{5(x-0.90)^2}{4})]\mathds{1}_{[0,1]}(x)$ and $F:= \int_\R f(x)dx$.
\end{itemize}

Figure \ref{Fig:Pos} shows the evolution of $t\mapsto (x_i^N(t))_{i\in\elts}$ and of $t\mapsto (m_i^N(t))_{i\in\elts}$ for $N=20$, $N=50$ and $N=100$. Due to the fact that the interaction function $\phi$ has compact support, we observe formation of clusters within the population. Note that as expected, the final number and positions of clusters are the same for all values of $N$ ($N$ big enough).
Within each cluster, the agents that are able to attract more agents gain influence (i.e. weight), while the followers tend to lose influence (weight).

Figure \ref{Fig:Mean-field} compares the evolutions of $t\mapsto \mu_t$ and $t\mapsto \mu^N_t$ at four different times. 
For visualization, the empirical measure was represented by the piece-wise constant counting measure $\tilde\mu^N_t$ defined by:
for all $x\in E_j$, $\tilde\mu^N_t(x) = \frac{p}{M} \sum_{i=1}^N m_i \mathds{1}\{x_i\in E_j\}$, where for each $j\in \{1,\cdots,p\}$, $E_j = [\frac{j-1}{p},\frac{j}{p})$, so that $(E_j)_{j\in\{1,\cdots,p\}}$ is a partition of $[0,1]$.
In Fig. \ref{Fig:Mean-field}, $p=41$.
We observe a good correspondence between the two solutions at all four time steps. 
Observe that the four clusters are formed at the same locations than in Figure \ref{Fig:Pos}, i.e. at $x=0.07$, $x=0.33$, $x=66$ and $x=0.9$. Convergence to the first and fourth clusters is slower than convergence to the second and third, due to the differences in the total weight of each cluster.

\begin{figure}
\centering
\includegraphics[trim = 1.1cm 0cm 1.4cm 0cm, clip=true, width=0.32\textwidth]{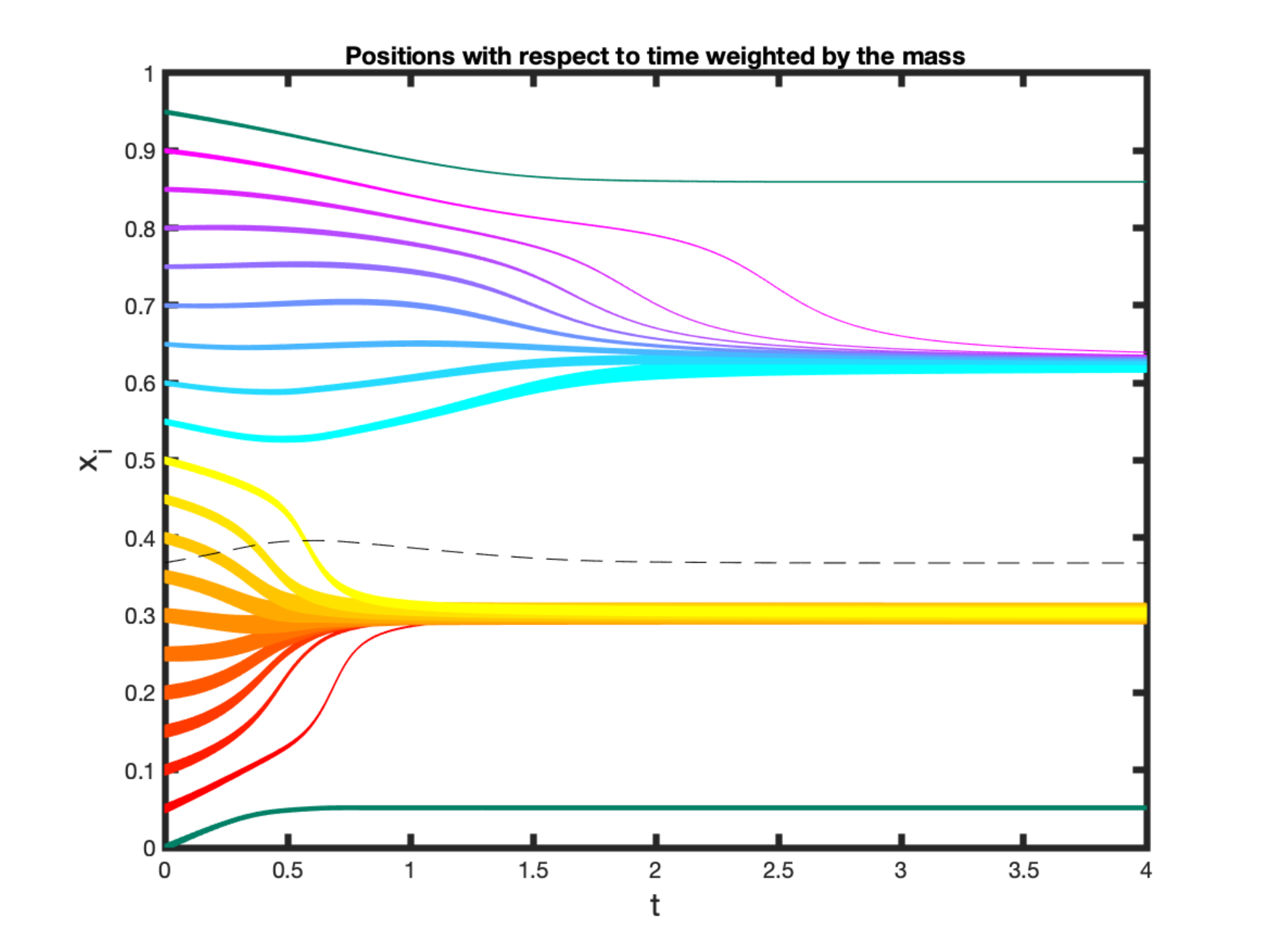}
\includegraphics[trim = 1.1cm 0cm 1.4cm 0cm, clip=true, width=0.32\textwidth]{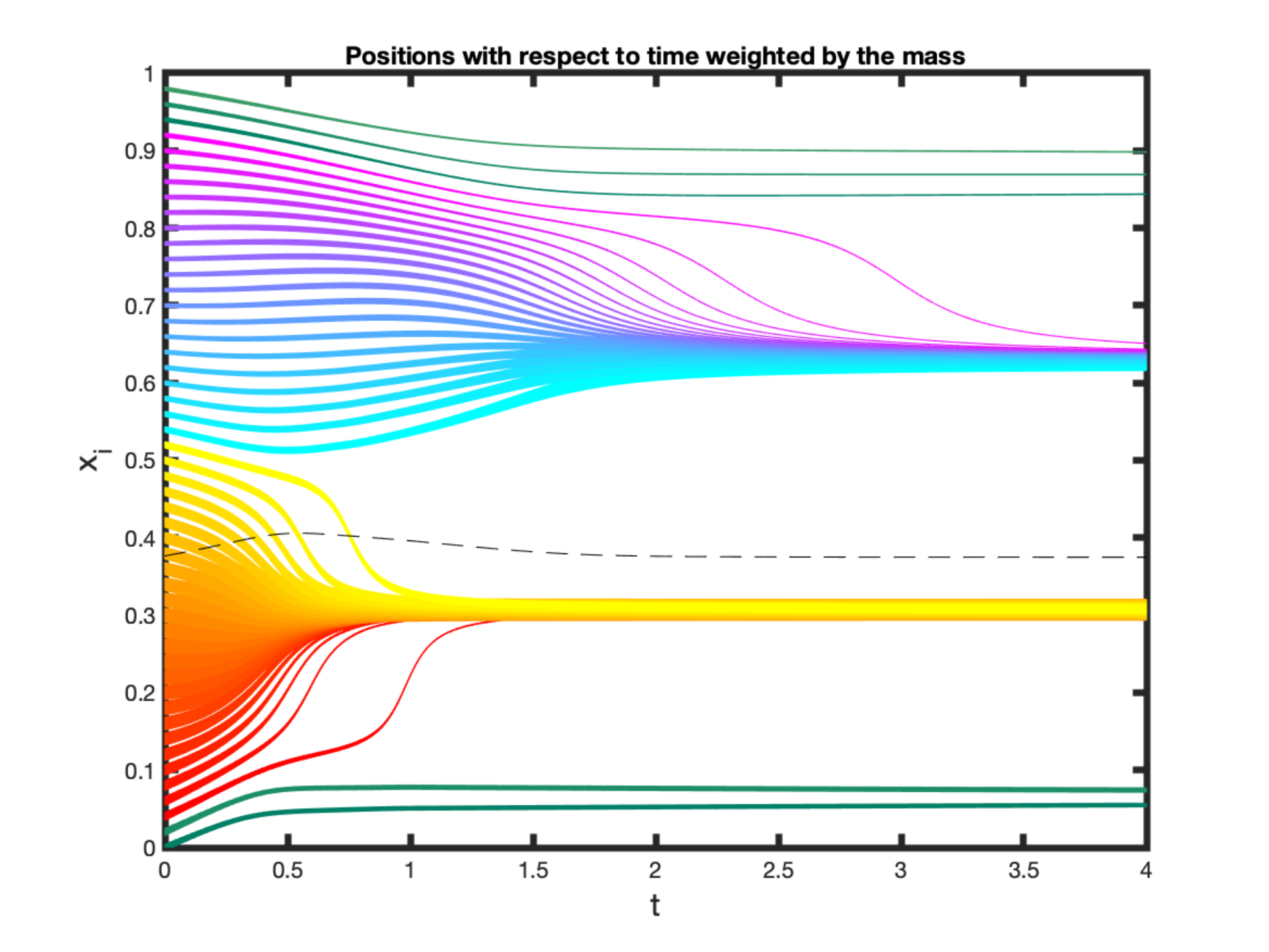}
\includegraphics[trim = 1.1cm 0cm 1.4cm 0cm, clip=true, width=0.32\textwidth]{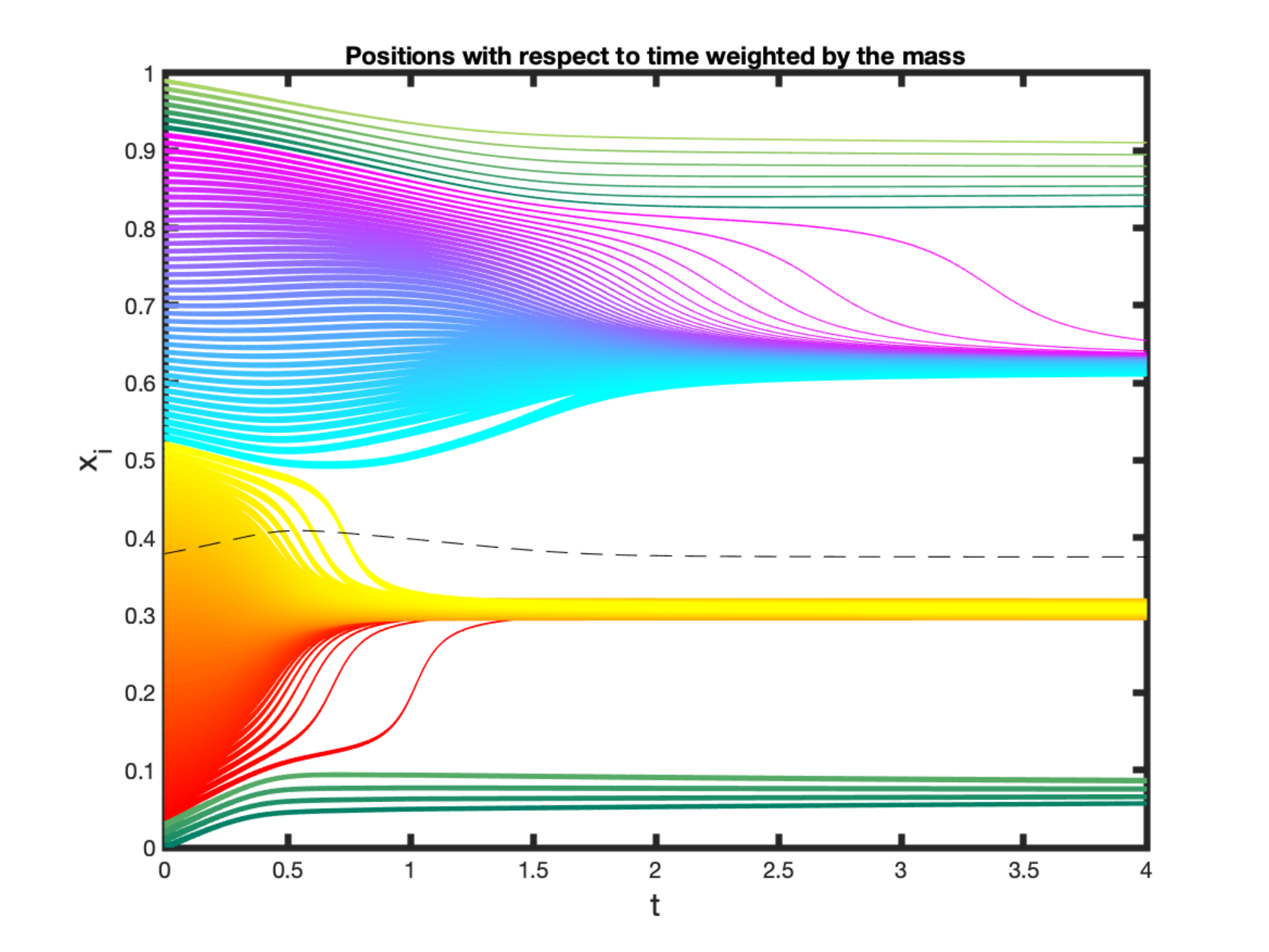}\\
\includegraphics[trim = 1.2cm 0.3cm 1.4cm 0cm, clip=true, width=0.32\textwidth]{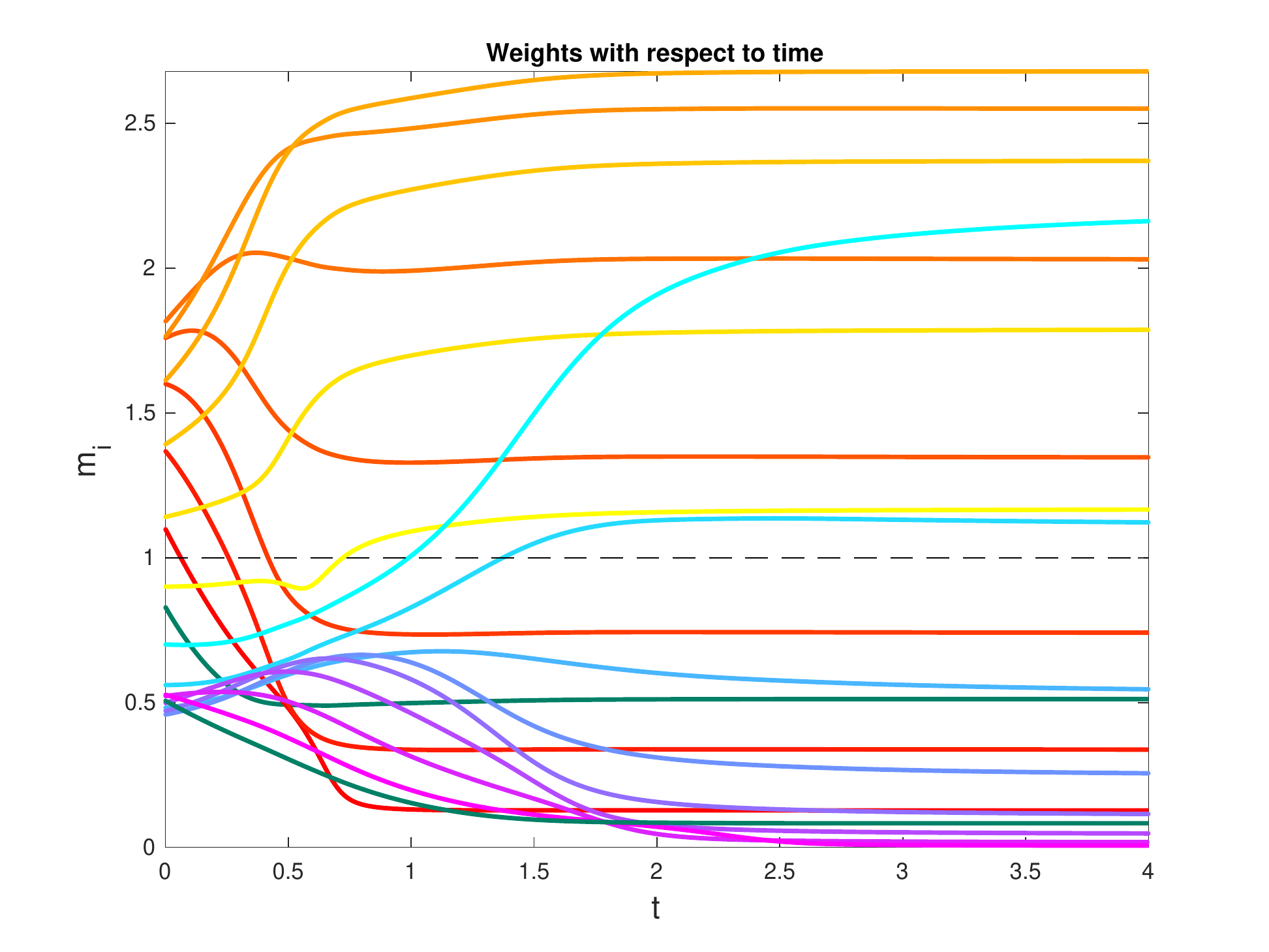}
\includegraphics[trim = 1.2cm 0.3cm 1.4cm 0cm, clip=true, width=0.32\textwidth]{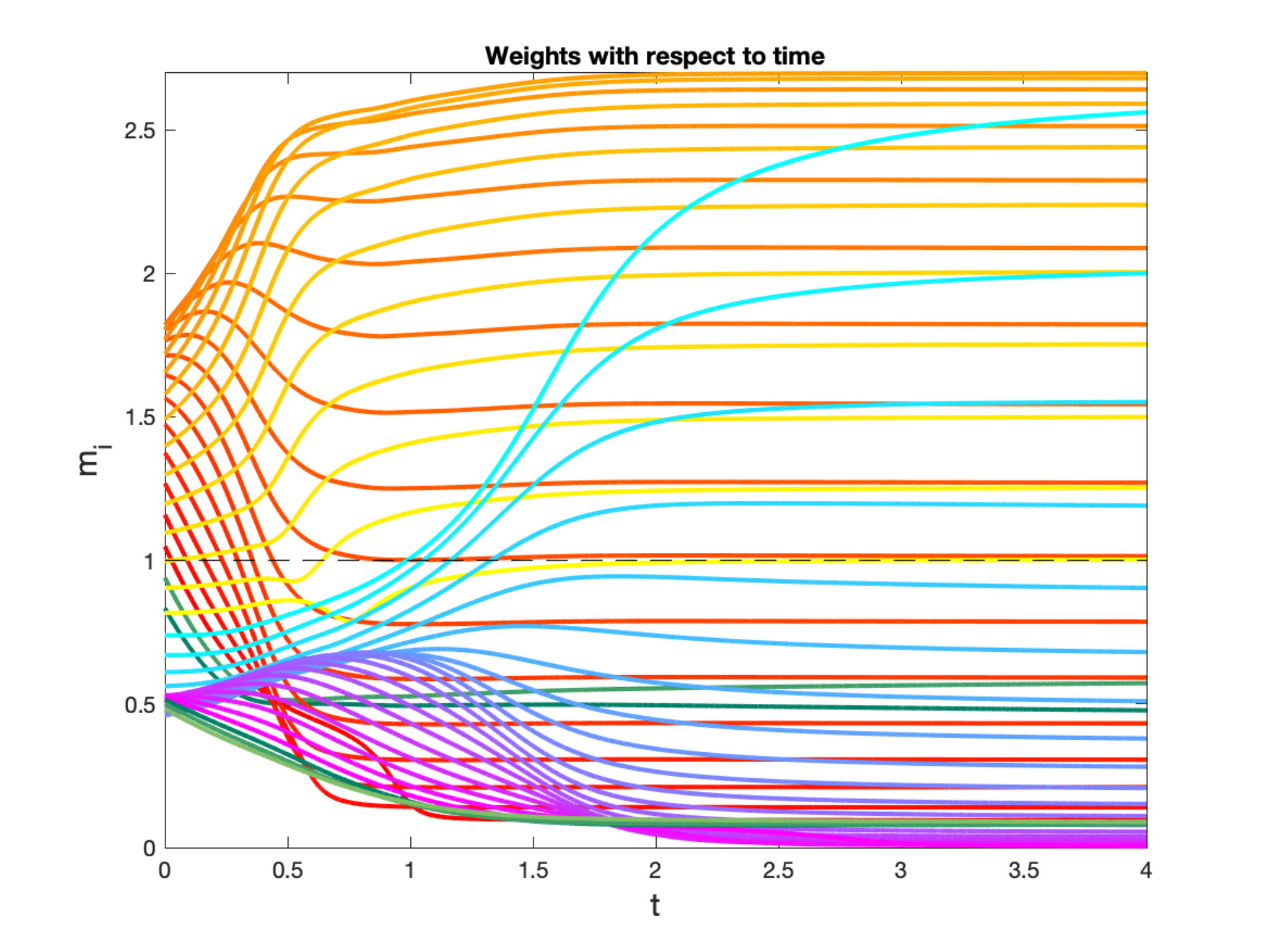}
\includegraphics[trim = 1.2cm 0.3cm 1.4cm 0cm, clip=true, width=0.32\textwidth]{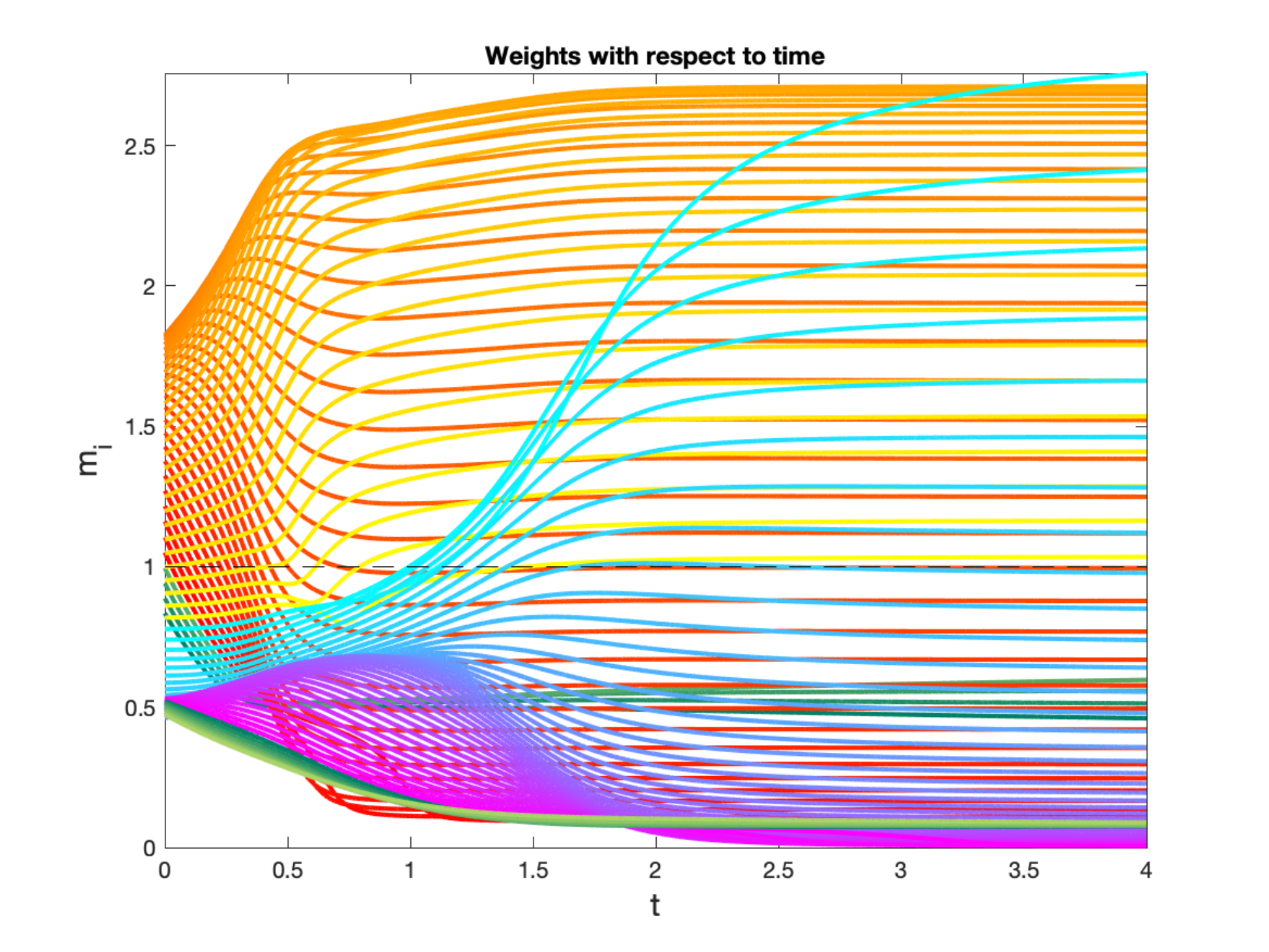}
\caption{Top row: Evolution of the positions for $N=20$, $N=50$ and $N=100$. The thickness of the lines is proportional to the agent's weight.
The dotted line represents the barycenter $\bar x:= \frac{1}{M}\sum_{i} m_i x_i$.
Bottom row: Evolution of the weights for $N=20$, $N=50$ and $N=100$. The dotted line represents the average weight $\bar m := \frac{1}{M}\sum_{i} m_i$.}
\label{Fig:Pos}
\end{figure}


\begin{figure}
\centering
\includegraphics[trim = 2.2cm 0.5cm 1.8cm 0cm, clip=true, width=0.34\textwidth]{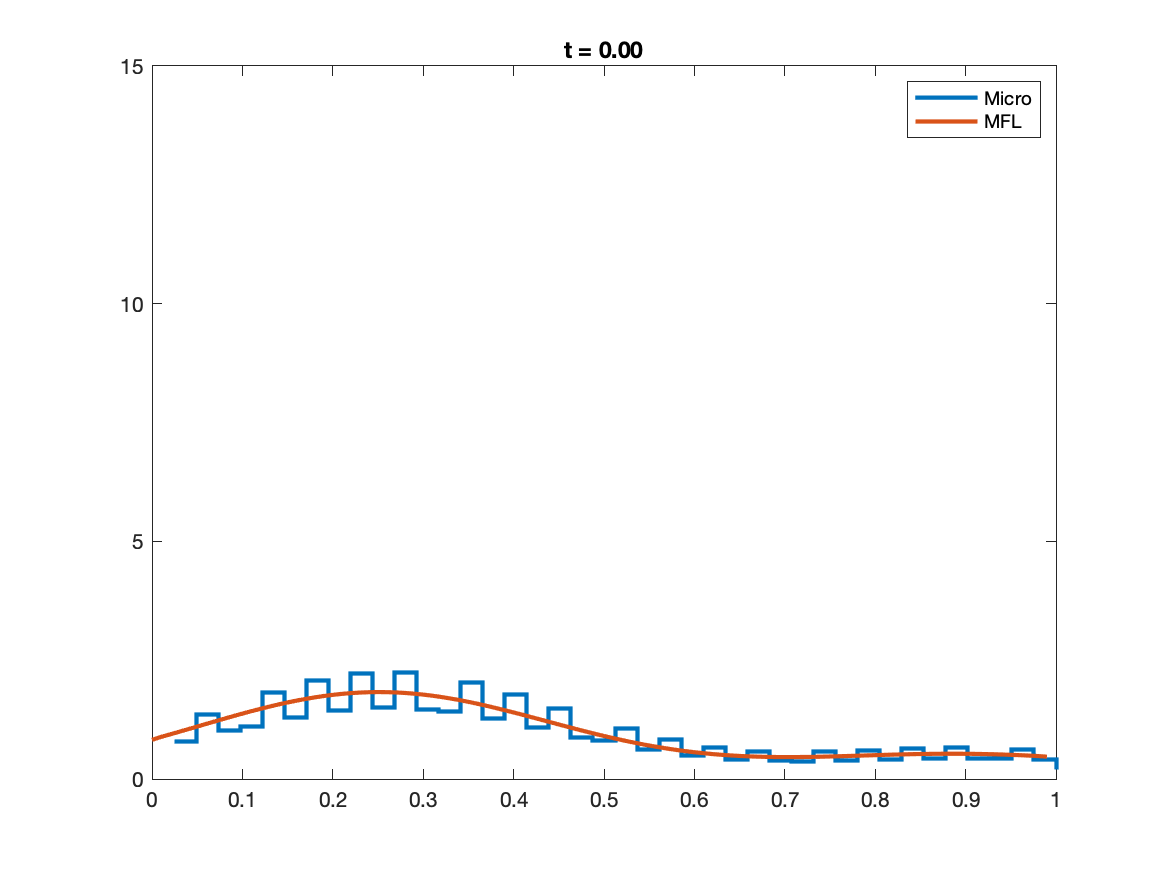}
\includegraphics[trim = 2.2cm 0.5cm 1.8cm 0cm, clip=true, width=0.34\textwidth]{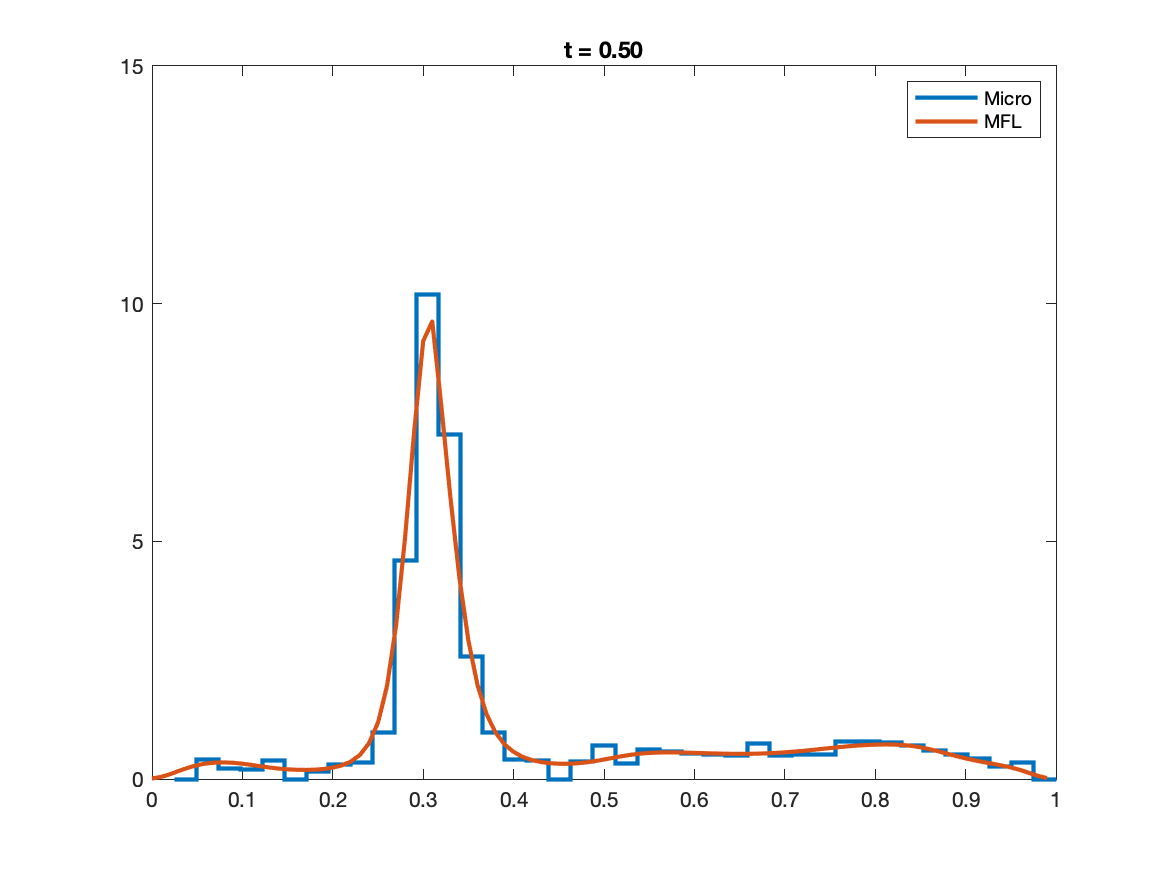}\\
\includegraphics[trim = 2.2cm 0.5cm 1.8cm 0cm, clip=true, width=0.34\textwidth]{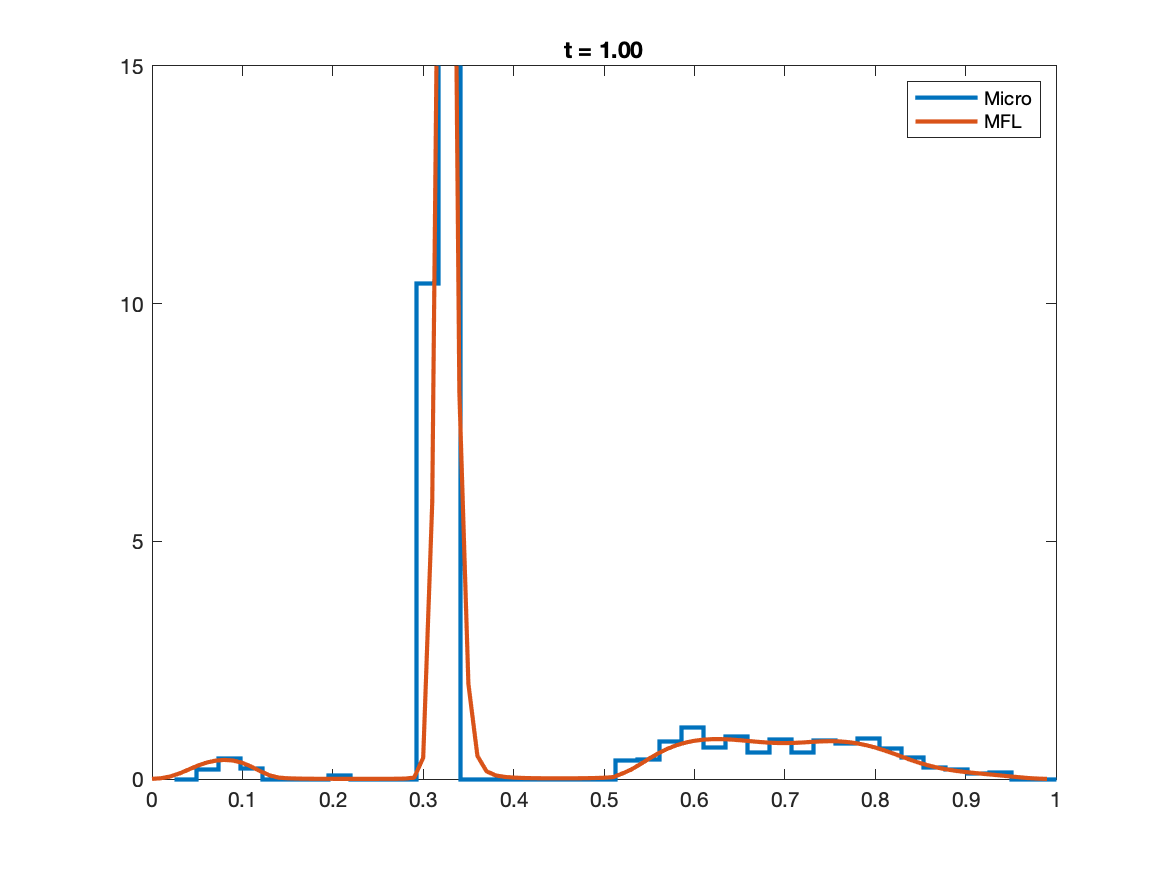}
\includegraphics[trim = 2.2cm 0.5cm 1.8cm 0cm, clip=true, width=0.34\textwidth]{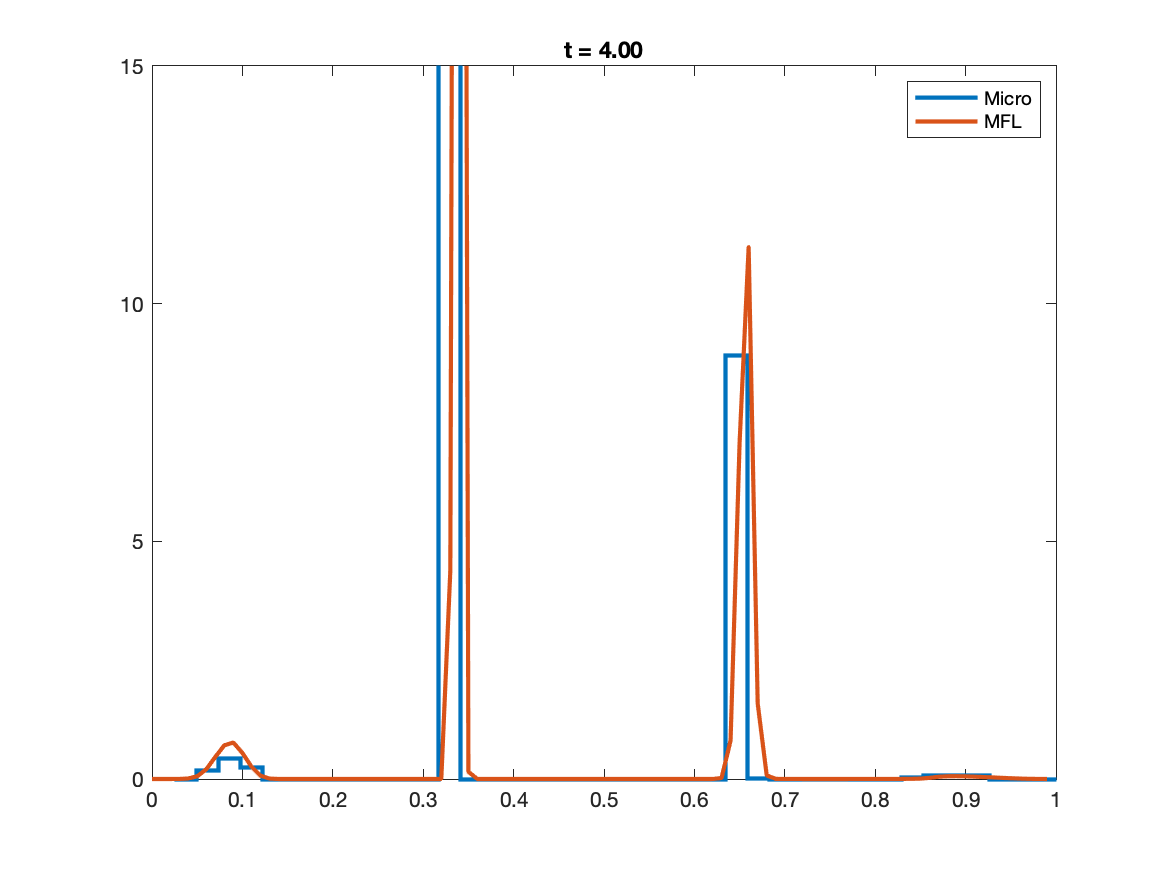}
\caption{Comparison of $\mu_t$ (in red), solution to the macroscopic model \eqref{eq:MFexample} and $\tilde{\mu}^N_t$ (in blue), counting measure corresponding to the solution to the microscopic model \eqref{eq:microexample} for $N=100$.} \label{Fig:Mean-field}
\end{figure}

\ifAppendix

\newpage
\appendix
\section{Appendix}
\subsection{Well-Posedness of the microscopic model}\label{Sec:ExUniqMicro}

We provide the proof of Proposition \ref{Prop:WellPosmicro}.
It is modeled after the proof of existence and uniqueness of the Graph Limit equation provided in \cite{AyiPouradierDuteil20}, but we write it fully here for self-containedness.

\begin{proof}
Let $\tx=(\tx_i)_{i\in\elts}\in C([0,T];(\R^d)^N)$ and $\tm=(\tm_i)_{i\in\elts}\in C([0,T];\R_+^N)$.
Consider the following decoupled systems of ODE:
\begin{equation}\label{eq:x-decoupled}
 \dot{x}_i(t) = \frac{1}{M} \sum\limits_{j=1}^N \tm_j(t) \ba(x_j(t)-x_i(t)), \quad x_i(0)=\xinit_i 
\end{equation}
and
\begin{equation}\label{eq:m-decoupled}
\displaystyle \dot{m}_i(t) = m_i(t)\frac{1}{M^q} \sum_{j_1=1}^N \cdots \sum_{j_q=1}^N m_{j_1}(t)\cdots m_{j_q}(t) S(\tx_i(t), \tx_{j_1}(t), \cdots \tx_{j_q}(t)), \quad m_i(0)=\minit_i.
\end{equation}

We begin by proving that there exists a unique solution to the Cauchy problem given by \eqref{eq:x-decoupled}.
Let $\tT>0$ and
let $\X:= \{x\in  C([0,\tT];(\R^d)^N)\; | \; x(t=0)=\xinit \}$. Consider the application
\[
\begin{array}{cccc}
\Kx : & \X & \rightarrow & \X \\
 & x & \mapsto & \Kx x
\end{array}
\]
where for all $t\in [0,\tT]$ and $i\in\elts$, $(\Kx x)_i(t) := \displaystyle \xinit_i + \int_0^t \frac{1}{M} \sumj \tm_j(\tau) \phi(x_j(\tau)-x_i(\tau)) d\tau$.

Let us  show that $\Kx$ is contracting for the norm $\|x\|_X :=   \sup_{t\in [0,\tT]} \max_{i\in\elts} \|x_i(t)\|$.
Let $x,y\in \X$. Then for all $t\in [0,\tT]$ and $i\in\elts$,
\[
\begin{split}
\|(\Kx x)_i - \Kx y)_i\| & = \left \| \int_0^t \frac{1}{M} \sumj \tm_j(\tau) \left[ \phi(x_j(\tau)-x_i(\tau))-\phi(y_j(\tau)-y_i(\tau))\right] d\tau \right \| \\
& \leq \int_0^t \frac{1}{M} \sumj \tm_j(\tau) L_\phi \left\| (x_j(\tau)-x_i(\tau))-(y_j(\tau)-y_i(\tau))\right\| d\tau \\
& \leq \int_0^t \frac{L_\phi}{M} \sumj \tm_j(\tau) (\| (x_j(\tau)-y_j(\tau)\|+ \|x_i(\tau)-y_i(\tau)\|) d\tau \\
& \leq  2\frac{L_\phi}{M} \int_0^t  \sumj \tm_j(\tau)  d\tau \|x-y\|_X \leq 2\frac{L_\phi}{M} \tT  \sup_{t\in[0, \tT]} ( \sumj \tm_j(t) ) \|x-y\|_X .
\end{split}
\]
Since $\tm$ is given, choosing
$\tT \leq [4 \frac{L_\phi}{M} \sup_{t\in[0, \tT]}  \sum_{j=1}^N \tm_j(t) ]^{-1}$ ensures that $\Kx$ is contracting on $[0,\tT]$.
By the Banach contraction principle, there exists a unique solution $x\in C([0,\tT],(\R^d)^N)$. We then take $x(\tT)$ as the initial condition, and the local solution can be extended to $[0,2\tT]$, and by repeating the same argument, to $[0,T]$.
Moreover, since the integrand is continuous as a map $[0,T]\rightarrow (\R^d)^N$, $x$ is continuously differentiable and $x\in C^1([0,T],(\R^d)^N)$.

We now show existence and uniqueness of the solution to the second decoupled system \eqref{eq:m-decoupled}. Let $m^0\in \R^N_+$ such that $\sum_{i=1}^N \minit_i = M$. 
Let $\M:= \{m\in C([0,\tT],\R_+^N) \; | \; m(t=0) =  \minit \text{ and } \sum_{i=1}^N m_i \equiv M\}$.
Consider the application
\[
\begin{array}{cccc}
\Km : & \M & \rightarrow & \M \\
 & m & \mapsto & \Km m
\end{array}
\]
where for all $t\in [0,\tT]$ and $i\in\elts$,
\[(\Km m)_i(t) := \displaystyle m_i^0 + \int_0^t m_i(\tau) \frac{1}{M^q} \sum_{j_1=1}^N\cdots \sum_{j_q=1}^N m_{j_1}(\tau)  \cdots m_{j_q}(\tau)  S(\tx_i(\tau), \tx_{j_1}(\tau), \cdots \tx_{j_q}(\tau))d\tau. \]
We show that $\Km$ is contracting for the norm $\|m\|_{\M}:=\frac{1}{M} \sup_{t\in [0,\tT]} \sum_{i=1}^N |m_i(t)|$.
Let $m,p\in \M$. It holds:
\begin{equation*}
\begin{split}
|(\Km m - \Km p)_i| =  &\left| \int_0^t \frac{1}{M^q} \left[ m_i   \sum_{j_1\cdots j_q} m_{j_1}  \cdots m_{j_q} - p_i  \sum_{j_1\cdots j_q} p_{j_1}  \cdots p_{j_q} \right]  S(\tx_i, \tx_{j_1}, \cdots \tx_{j_q})d\tau \right| \\
 \leq & \int_0^t \frac{1}{M^q} |m_i-p_i|   \sum_{j_1\cdots j_q} m_{j_1}  \cdots m_{j_q}  |S(\tx_i, \tx_{j_1}, \cdots \tx_{j_q})| d\tau \\
 & + \int_0^t \frac{1}{M^q} p_i   \sum_{j_1\cdots j_q} | m_{j_1}-p_{j_1}| m_{j_2}  \cdots m_{j_q}  |S(\tx_i, \tx_{j_1}, \cdots \tx_{j_q})| d\tau \\
 & +\cdots + \int_0^t \frac{1}{M^q} p_i   \sum_{j_1\cdots j_q} p_{j_1} \cdots p_{j_{q-1}}  | m_{j_q}-p_{j_q}|  |S(\tx_i, \tx_{j_1}, \cdots\tx_{j_q})| d\tau  \\
 \leq & \; \bS \tT \sup_{[0,\tT]} |m_i-p_i| + q\, \bS \tT \frac{1}{M} \sup_{[0,\tT]} ( p_i \sum_{j=1}^N |m_j-p_j| ) \\
 \leq & \; \bS \tT \sup_{[0,\tT]} |m_i-p_i| + q\, \bS \tT  \sup_{[0,\tT]} \sum_{j=1}^N |m_j-p_j| .
\end{split}
\end{equation*}
Thus,
$\|\Km m - \Km p\|_{\M} \leq (q+1)\,\bS \tT \|m-p\|_{\M}$.
Taking $\tT\leq \frac{1}{2} ((q-1) \bS)^{-1}$, the operator $\Km$ is contracting. By the same reasoning as previously, there is a unique solution $m\in C^1([0,T],\R_+^N)$ to \eqref{eq:m-decoupled}.

Let us define the sequences $(x^n){n\in\N}$ and $(m^n){n\in\N}$ by 
\begin{itemize}
\item For all $t\in [0,T]$, $m^0(t)=\minit$ and $x^0(t)=\xinit$
\item For all $n \geq 1$, $x^n$ and $m^n$ are solutions to the system of ODE
\begin{equation*}
\begin{cases}
\displaystyle \dot{x}^n_i = \frac{1}{M} \sum\limits_{j=1}^N m^{n-1}_j \ba(x^n_j-x^n_i), \quad x_i^n(0)=\xinit_i  \\
\displaystyle \dot{m}^n_i = m_i^n\frac{1}{M^q} \sum_{j_1=1}^N \cdots \sum_{j_q=1}^N m_{j_1}^n\cdots m_{j_q}^n S(x_i^{n-1}, x_{j_1}^{n-1}, \cdots x_{j_q}^{n-1}), \quad m_i^n(0)=\minit_i
\end{cases}
\end{equation*}
\end{itemize}
The results obtained above ensure that the sequences are well defined and that for all $n\in\N$, $(x^n,m^n)\in C([0,T];(\R^d)^N\times\R_+^N)$.
We begin by showing that $x^n$ and $m^n$ are bounded in $L^\infty$ norm independently of $n$.
It holds:
\[
\begin{split}
|m_i^n(t)| & = \left| \minit_i + \int_0^t m_i^n(\tau)\frac{1}{M^q} \sum_{j_q=1}^N m_{j_1}^n(\tau)\cdots m_{j_q}^n(\tau) S(x_i^{n-1}(\tau), x_{j_1}^{n-1}(\tau), \cdots x_{j_q}^{n-1}(\tau)) d\tau \right|\\
& \leq |\minit_i| + \bS \int_0^t |m_i^n(\tau)| d\tau.
\end{split}
\]
From Gronwall's lemma, for all $t\in [0,T]$, $|m_i^n(t)|\leq \minit_i e^{\bS t} \leq M_T$
where $M_T:= \max_{i\in\elts} \minit_i e^{\bS T}$.

Similarly, notice that for all $z\in\R^d$, $\|\phi(z)\|\leq \Phi_0 + L_\phi \|z\|$,where $\Phi_0=\phi(0)$.
Then
\[
\begin{split}
\|x^n_i(t)\| = & \left\| \xinit_i + \frac{1}{M} \int_0^t \sum\limits_{j=1}^N m^{n-1}_j(\tau) \ba(x^n_j(\tau)-x^n_i(\tau)) d\tau \right\| \\
\leq & \|\xinit_i\| +  \frac{1}{M} \int_0^t \sum\limits_{j=1}^N m^{n-1}_j(\tau) (\Phi_0+ L_\phi \|x^n_j(\tau)-x^n_i(\tau)\|) d\tau  \\
\leq & \|\xinit_i\| +  \frac{M_T}{M} \int_0^t \sum\limits_{j=1}^N  (\Phi_0+ 2L_\phi \max_{i\in\elts}\|x^n_i(\tau)\|) d\tau .
\end{split}
\]
Thus
\[
\max_{i\in\elts} \|x_i^n(t)\|\leq  \max_{i\in\elts} \|\xinit_i \| + \frac{M_T}{M} ( \Phi_0 t + 2 L_\phi \int_0^t \max_{i\in\elts} \|x_i^n(\tau)\| d\tau)
\]
and from Gronwall's lemma, for all $t\in [0,T]$,
\[
\max_{i\in\elts} \|x_i^n(t)\| \leq X_T:= \left[ \max_{i\in\elts}\|\xinit_i\|  + \frac{M_T}{M} \Phi_0 T\right] e^{2 L_\phi \frac{M_T}{M} T}
\]
We prove that $(x^n)_{n\in\N}$ and $(m^n)_{n\in\N}$ are Cauchy sequences. For all $n\in\N$,
\[
\begin{split}
& \|x_i^{n+1}-x_i^n\| =   \left\| \int_0^t\frac{1}{M} \sum_{j=1}^N m_j^n\phi(x_j^{n+1}-x_i^{n+1}) d\tau - \int_0^t\frac{1}{M} \sum_{j=1}^N m_j^{n-1}\phi(x_j^{n}-x_i^{n}) d\tau  \right\|  \\
& =  \left\| \int_0^t\frac{1}{M} \sum_{j=1}^N (m_j^n-m_j^{n-1})\phi(x_j^{n+1}-x_i^{n+1}) d\tau + \int_0^t\frac{1}{M} \sum_{j=1}^N m_j^{n-1}[\phi(x_j^{n+1}-x_i^{n+1})-\phi(x_j^{n}-x_i^{n})] d\tau  \right\| \\
& \leq \int_0^t \frac{1}{M} \sum_{j=1}^N |m_j^n-m_j^{n-1}| (\Phi_0+ L_\phi \|x_j^{n+1}-x_i^{n+1}\|) d\tau
+ \int_0^t\frac{M_T L_\phi}{M} \sum_{j=1}^N \| x_j^{n+1}-x_j^{n}+x_i^{n+1}-x_i^{n}\| d\tau  \\
& \leq \frac{1}{M} (\Phi_0+ 2 L_\phi X_T) \int_0^t \sum_{j=1}^N |m_j^n-m_j^{n-1}| d\tau
+ \frac{M_T L_\phi}{M} \int_0^t \sum_{j=1}^N (\| x_j^{n+1}-x_j^{n}\|+\|x_i^{n+1}-x_i^{n}\|) d\tau
\end{split}
\]
Thus
\[
\sum_{i=1}^N \|x_i^{n+1}-x_i^n\| \leq 
 \frac{N}{M} (\Phi_0+ 2 L_\phi X_T) \int_0^t \sum_{i=1}^N |m_i^n-m_i^{n-1}| d\tau
+ 2N \frac{M_T L_\phi}{M} \int_0^t \sum_{i=1}^N \| x_i^{n+1}-x_i^{n}\| d\tau.
\]

A similar computation, for $m$ gives
\[
\begin{split}
|m_i^{n+1}-m_i^n| = & \bigg | \int_0^t m_i^{n+1} \frac{1}{M^q} \sum_{j_1\cdots j_q} m_{j_1}^{n+1}  \cdots m_{j_q}^{n+1}  S(x_i^{n}, x_{j_1}^{n},  \cdots x_{j_q}^{n})d\tau \\
& - \int_0^t m_i^n \frac{1}{M^q} \sum_{j_1\cdots j_q} m_{j_1}^n  \cdots m_{j_q}^n  S(x_i^{n-1}, x_{j_1}^{n-1}, 
\cdots x_{j_q}^{n-1})d\tau \bigg |\\
\leq &
\int_0^t |m_i^{n+1}-m_i^{n}| \frac{1}{M^q} \sum_{j_1\cdots j_q} m_{j_1}^{n+1}  \cdots m_{j_q}^{n+1}  S(x_i^{n}  \cdots x_{j_q}^{n})d\tau \\
& + \int_0^t m_i^{n} \frac{1}{M^q} \sum_{j_1\cdots j_q} |m_{j_1}^{n+1}-m_{j_1}^{n}| m_{j_2}^{n+1} \cdots m_{j_q}^{n+1}  S(x_i^{n}  \cdots x_{j_q}^{n})d\tau \\
& + \cdots + \int_0^t m_i^{n} \frac{1}{M^q} \sum_{j_1\cdots j_q} m_{j_1}^{n} \cdots m_{j_{q-1}}^{n}  |m_{j_q}^{n+1}-m_{j_q}^{n}|  S(x_i^{n}  \cdots x_{j_q}^{n})d\tau  \\
& + \int_0^t m_i^{n} \frac{1}{M^q} \sum_{j_1\cdots j_q} m_{j_1}^{n} \cdots m_{j_q}^{n} | S(x_i^{n}  \cdots x_{j_q}^{n}) - S(x_i^{n-1}  \cdots x_{j_q}^{n-1})| d\tau \\
\end{split}
\]
From \eqref{eq:psiSklip}, it holds
\[
\begin{split}
& \int_0^t m_i^{n} \frac{1}{M^q} \sum_{j_1\cdots j_q} m_{j_1}^{n} \cdots m_{j_q}^{n} | S(x_i^{n}  \cdots x_{j_q}^{n}) - S(x_i^{n-1}  \cdots x_{j_q}^{n-1})| d\tau \\
\leq &
\int_0^t m_i^{n} \frac{1}{M^q} \sum_{j_1\cdots j_q} m_{j_1}^{n} \cdots m_{j_q}^{n} L_S ( \|x_i^{n}-x_i^{n-1} \|+ \| x_{j_1}^{n}- x_{j_1}^{n-1}\|  \cdots + \| x_{j_q}^{n}- x_{j_q}^{n-1}\| ) d\tau \\
\leq & \int_0^t m_i^{n} L_S \|x_i^{n}-x_i^{n-1} \| d\tau + q \int_0^t \frac{1}{M} \sum_{j=1}^N m_{j}^{n} L_S  \|x_j^{n}-x_j^{n-1} \| d\tau. \\
\end{split}
\]
Thus, 
\[
\begin{split}
|m_i^{n+1}-m_i^n| 
\leq & \bS \int_0^t |m_i^{n+1}-m_i^{n}| d\tau  + q \bS \frac{M_T}{M} \int_0^t \sum_{j=1}^N |m_j^{n+1}-m_j^{n}| d\tau \\
& + M_T \int_0^t L_S \|x_i^{n}-x_i^{n-1} \| d\tau + q  L_S \frac{M_T}{M} \int_0^t  \sum_{j=1}^N  \|x_j^{n}-x_j^{n-1} \| d\tau.
\end{split}
\]
Summing the terms, it holds
\[
\begin{split}
\sumi |m_i^{n+1}-m_i^n| 
\leq & \bS(1+ q N \frac{M_T}{M}) \int_0^t \sum_{j=1}^N |m_j^{n+1}-m_j^{n}| d\tau 
 + M_T L_S(1+ q \frac{N}{M}) \int_0^t  \sum_{i=1}^N  \|x_j^{n}-x_j^{n-1} \| d\tau.
\end{split}
\]

Summarizing, we have
\[
\sum_{i=1}^N \|x_i^{n+1}-x_i^n\| \leq 
 C_1 \int_0^t \sum_{i=1}^N \| x_i^{n+1}-x_i^{n}\| d\tau
+C_2 \int_0^t \sum_{i=1}^N |m_i^n-m_i^{n-1}| d\tau.
\]
and
\[
\sumi |m_i^{n+1}-m_i^n| 
\leq  C_3 \int_0^t \sum_{i=1}^N |m_i^{n+1}-m_i^{n}| d\tau + C_3 \int_0^t  \sum_{i=1}^N  \|x_i^{n}-x_i^{n-1} \| d\tau.
\]
where $C_1=2N \frac{M_T L_\phi}{M}$, $C_2=\frac{N}{M} (\Phi_0+ 2 L_\phi X_T)$, $C_3= \bS(1+ q N \frac{M_T}{M})$ and $C_4=M_T L_S(1+ q \frac{N}{M}) $.
Let $u_n:= \sum_{i=1}^N \|x_i^{n+1}-x_i^n\|+\sum_{i=1}^N |m_i^{n+1}-m_i^{n}| $ for all $n\in\N$.
Then 
\[
u_n(t) \leq A_T \int_0^t u_n(\tau) d\tau + A_T \int_0^t u_{n-1}(\tau) d\tau
\]
where $A_T:= \max(C_1,C_2,C_3,C_4)$.
From Gronwall's lemma, for all $t\in [0,T]$,
\[
u_n(t) \leq A_T e^{A_T T} \int_0^t u_{n-1}(\tau) d\tau 
\]
which, by recursion, implies
\[
u_n(t) \leq \frac{(A_T e^{A_T T})^n}{n!} \sup_{[0,T]} u_0.
\]
This is the general term of a convergent series.
Thus, for all $n,p\in\N$, 
\[
\sumi \| x_i^{n+p}-x_i^n \| \leq  \sum_{k=n}^{n+p-1} \sumi \|x_i^{k+1}-x_i^k\| \leq \sum_{k=n}^{n+p-1} u_k \xrightarrow[n,p\rightarrow+\infty]{} 0.
\]
This proves that $(x^n)_{n\in\N}$ is a Cauchy sequence in the Banach space $C([0,T], (\R^d)^N)$ for the norm $x\mapsto \sup_{t\in[0,T]}\sum_{i=1}^N \|x_i^n(t)\|$. Similarly, $(m^n)_{n\in\N}$ is a Cauchy sequence in $C([0,T], \R_+^N)$ for the norm $m\mapsto \sup_{t\in[0,T]} \sum_{i=1}^N |m_i^n(t)|$.
One can easily show that their limits $(x,m)$ satisfy the system of ODEs \eqref{eq:syst-micro-intro}. Furthermore, since the bounds $X_T$ and $M_T$ do not depend on $n$, it holds $\|x_i(t)\|\leq X_T$ and $|m_i(t)|\leq M_T$ for all $t\in [0,T]$ and every $i\in\elts$.
This concludes the proof of existence.

Let us now deal with uniqueness.
Suppose that $(x,m)$ and $(p,m)$ are two couples of solutions to the Cauchy problem \eqref{eq:syst-micro-intro} with the same initial conditions $(\xinit, \minit)$.
Similar computations to the ones done previously give
\[
\sumi \|x_i(t)-y_i(t)\| + \sumi |m_i(t)-p_i(t)| \leq A_T \int_0^t (\sumi \|x_i(\tau)-y_i(\tau)\| + \sumi |m_i(\tau)-p_i(\tau)| ) d\tau
\]
By Gronwall's lemma
\[
\sumi \|x_i(t)-y_i(t)\| + \sumi |m_i(t)-p_i(t)| \leq ( \sumi \|x_i(0)-y_i(0)\| + \sumi |m_i(0)-p_i(0)| ) e^{A_T t} = 0
\]
which concludes uniqueness.

\end{proof}

\subsection{Properties of another numerical scheme} \label{Sec:SchemePRT}

Here, we show that the numerical scheme $\tilde{\S}$ introduced in \cite{PiccoliRossiTournus20} preserves neither mass, nor positivity, nor total variation.\\
We remind the definition of the scheme $\tilde{\S}$. For all $k\in\N$, let $\dt:=2^{-k}T$ and $\mu_0^k=\mu_0$. For all $i\in\{0,\cdots, 2^k-1\}$, $\mu^k_{(i+1)\dt}$ is defined from $\mu^k_{i\dt}$ as:
\[
\mu^k_{(i+1)\dt} = \phi^{V[\mu_{i\dt}^k]}\#\mu_{i\dt}^k + \dt \, h[\mu^k_{i\dt}].
\]
Let $d=1$, and
consider a velocity field $V[\mu]$ and a source term $h[\mu]$ defined by: for all $\mu\in\P(\R)$, for all $x\in\R$,
\[
V[\mu](x)=\mu(\R)\, v(x):=\left\{\begin{array}{ll}
\mu(\R)\,\sign(x) \quad & \text{ for } |x|\geq 1 \\
\mu(\R)\, x \quad & \text{ for } |x|< 1 
\end{array}
\right.
\quad 
\text{ and }
\quad
h[\mu](x) = \left(\int_\R (x-y) d\mu(y)\right)\mu(x).
\]
The vector field is a slightly modified version of \eqref{eq:Vdef}, where $V[\mu](x) = \int_\R \phi(x,y)d\mu(y)$, with $\phi(x,y) = v(x)$.
One can easily show that on a fixed time interval $[0,T]$, $h$ satisfies the assumptions given in Hyp. \ref{hyp:S}. Then from Theorem \ref{Th:existenceuniqueness}, if $\mu_0\in\P_c(\R)$, we expect the solution $\mu_t$ to the transport PDE with source \eqref{eq:transportsource} to remain a probability measure at all time.\\
Let $\mu_0:=\frac{1}{2}(\delta_{1}+\delta_{-1})\in \P_c(\R).$ Initially, the center of mass is $\int_\R yd\mu_0(y)=0$.
One can compute the evolution of $\mu_0$ explicitely at each time step $i\dt$.
For $i=1$,
\[\mu^k_{\dt} = \frac{1}{2}(\delta_{1+\dt}+\dt\, \delta_{1}-\dt\,\delta_{-1}+\delta_{-1-\dt}).\]
Hence at the first time step, positivity is already lost. We notice that at this stage, the total mass is conserved as $\mu^k_{\dt}(\R)=\mu^k_{0}(\R)=1$, but the total variation is not: $|\mu^k_{\dt}| = 1+\dt$. The center of mass is not conserved either:  $\int_\R yd\mu^k_{\dt}(y)=\dt$.\\
For $i=2$,
\[\mu^k_{2\dt} =  \frac{1}{2}(\delta_{1+2\dt}+2\dt\,\delta_{1+\dt}+\dt^2(1-\dt)\, \delta_{1} +  \dt^2(1+\dt)\, \delta_{-1} - 2\dt(1+\dt)\,\delta_{-1-\dt}+\delta_{-1-2\dt}).\]
Again, it holds $\mu^k_{2\dt}(\R)=1$, but none of the other quantities are conserved: $|\mu^k_{2\dt}|=1+2(\dt+\dt^2)$ and $\int_\R yd\mu^k_{\dt}(y)=2\dt+3\dt^2$. \\
Observe that this numerical scheme also has a dispersive effect, due to the simultaneous treatment of the transport and source operators. Whereas the transport term correctly transports the Dirac masses initially located at $1$ and $-1$, the source term creates new Dirac masses along their trajectory. Hence the solution $\mu^k_{n\dt}$ to the scheme at time $n\dt$ is composed of $2(n+1)$ Dirac masses instead of the two Dirac masses composing the exact solution to the PDE.

For comparison, we provide the evolution of $\mu_0$ obtained with the scheme $\S$ defined in Section \ref{Sec:numscheme}, starting with $\mu_0= \frac{1}{2}(\delta_1+\delta_{-1})$. 
It holds
\[
\begin{split}
& \mu^k_{\dt} =  \frac{1}{2}((1+\dt)\delta_{1+\dt}+(1-\dt)\delta_{-1-\dt})\\
& \mu^k_{2\dt} =  \frac{1}{2}((1+2\dt+\dt^2-\dt^3-\dt^4)\delta_{1+\dt}+(-1-2\dt-\dt^2+\dt^3+\dt^4)\delta_{-1-2\dt}).
\end{split}
\]
Observe that positivity is preserved, as well as the total mass $\mu(\R^d)$.
Figure \ref{Fig:SchemeExamples} illustrates schematically the evolutions of $\mu_0$ by $\S$ and $\tilde{\S}$.
\begin{figure}
\centering
\includegraphics[scale=1]{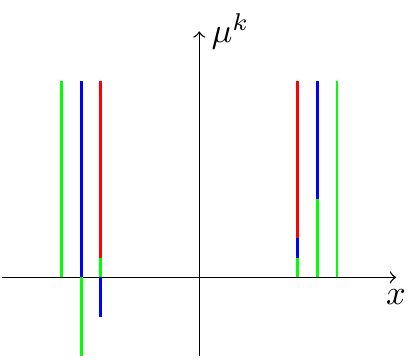} \quad
\includegraphics[scale=1]{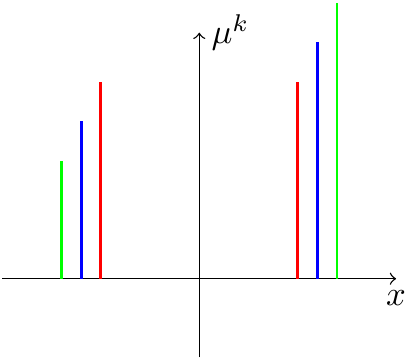} 
\caption{Comparison of the evolutions of $\mu$ by the two numerical schemes $\tilde{\S}$ (left) and $\S$ (right) at time $t=0$ (red), $t=\dt$ (blue) and $t=2\dt$ (green). Positivity is not preserved in the scheme $\tilde{\S}$. On the other hand, with $\S$, $\mu^k$ remains a probability measure at all time.}\label{Fig:SchemeExamples}
\end{figure}

\fi

\newpage
\bibliography{biblio}

\balance

\end{document}